\documentclass[12pt]{amsart}
\usepackage[margin=1.25in,letterpaper,portrait]{geometry}
\usepackage{latexsym,amssymb,verbatim,multirow,graphicx,epsfig,enumerate, paralist,hyperref}
\usepackage{xcolor}
\usepackage[normalem]{ulem}

\theoremstyle{plain}
   \newtheorem{main}{Main Theorem}
   
   \newtheorem{theorem}{Theorem}[section]
   \newtheorem{proposition}[theorem]{Proposition}
   \newtheorem{prop}[theorem]{Proposition}

    \newtheorem{cor}[theorem]{Corollary}
   \newtheorem{conjecture}[theorem]{Conjecture}
   
   \newtheorem*{theorem*}{Theorem}
   \newtheorem{propdef}[theorem]{Proposition/Definition}
\theoremstyle{definition}

   \newtheorem{remark}[theorem]{Remark}

\numberwithin{equation}{section}

\newcommand\qbin[3]{\left[\begin{matrix} #1 \\ #2 \end{matrix} \right]_{#3}}
\newcommand\hook[2]{{\langle #2-#1, 1^{#1} \rangle}}

\newcommand\Symm{\mathfrak{S}}

\newcommand{\vep}{\varepsilon}

\newcommand\one{\mathbf{1}}

\newcommand\Par{\operatorname{Par}}

\newcommand\Irr{\operatorname{Irr}}

\newcommand\wt{\widetilde}
\newcommand{\wh}{\widehat}
\newcommand\fix{\operatorname{fix}}

\newcommand\rk{\operatorname{rk}}

\newcommand{\Fac}{\mathcal{F}}
\newcommand{\Grfn}{\mathcal{Q}}

\newcommand{\defn}[1]{\textbf{#1}}

\newcommand\TTT{{\mathcal{T}}}

\newcommand\llambda{{\vec{\lambda}}}

\newcommand\CC{{\mathbb{C}}}

\newcommand\FF{{\mathbb{F}}}
\newcommand\GG{{\mathbf{G}}}
\newcommand{\KK}{\mathbb{K}}
\newcommand\GL{{\mathrm{GL}}}
\newcommand{\GLnFq}{\GL_n(q)}
\newcommand\SL{{\mathrm{SL}}}
\newcommand{\SLnFq}{\SL_n(q)}
\newcommand\GU{{\mathrm{GU}}}
\newcommand{\GUnFq}{\GU_n(q)}
\newcommand\SU{{\mathrm{SU}}}
\newcommand{\SUnFq}{\SU_n(q)}
\newcommand{\GLU}{\mathrm{GL}^\vep}
\newcommand{\SLU}{\mathrm{SL}^\vep}
\newcommand{\Frob}{F_\vep}

\newcommand{\rootorbs}{\Phi_\vep}
\newcommand{\charorbs}{\wh{\Phi}_\vep}

\newcommand{\Fq}{\FF_q}
\newcommand{\op}{\operatorname}
\newcommand{\ol}{\overline}
\newcommand{\bmat}[1]{\begin{bmatrix} #1 \end{bmatrix}}
\newcommand{\im}{\op{im}}
\newcommand{\mov}{\op{mov}}
\newcommand{\res}{\big|}

\title[Counting factorizations of Singer cycles]{Counting factorizations of Singer cycles in linear and unitary groups}
\author{Joel Brewster Lewis}
\author{C. Ryan Vinroot}

\begin{document}

\begin{abstract}
We count factorizations of Singer cycles as products of reflections in the families of special and general unitary and linear groups over a finite field.  In the case of minimum-length factorizations, the resulting answer is a striking product formula resembling the count for minimum-length factorizations of Coxeter elements into reflections in complex reflections groups.  Moreover, for minimum length, the answers for the unitary and linear groups exhibit the phenomenon of Ennola duality, where the number of factorizations in a unitary group over the field $\Fq$ is given by replacing `$q$' with `$-q$' in the corresponding answer for a linear group.  We use the character theory of these groups to make this count, and in particular we employ the Deligne--Lusztig theory of characters for finite reductive groups.
\end{abstract}

\maketitle

\tableofcontents

This paper is motivated by the following results in the field of Coxeter--Catalan combinatorics.
\begin{theorem*}[\cite{Hurwitz, Denes, Chapoton, BessisKpi1}]
    The number of factorizations of an $n$-cycle in the symmetric group $\Symm_n$ as a product of the minimum number $n - 1$ of transpositions is $n^{n - 2}$.
    More generally, the number of factorizations of a Coxeter element $c$ in an irreducible finite real or complex reflection group $W$ of rank $N$ as a product of the minimum number $N$ of reflections is $\dfrac{N! \cdot h^N}{|W|}$, where $h = |c|$ is the order of $c$.
\end{theorem*}
(The first half of the theorem may be recovered from the second half by the substitutions $|\Symm_n| = n!$, $\rk(\Symm_n) = n - 1$, and $h = n$.)
In \cite{LRS}, an analogous theorem was proved for the general linear group over the finite field $\Fq$, which may be viewed as a reflection group in positive characteristic.  In this theorem, the role of the \emph{Coxeter element} is played by a \emph{Singer cycle}---see Section~\ref{sec:background} for definitions of this and other technical terms.
\begin{theorem*}[{\cite[Thm.~1.1]{LRS}}]
  The number of factorizations of a Singer cycle in the general linear group $\GLnFq$ over the finite field $\Fq$ as a product of the minimum number $n$ of reflections is $(q^n - 1)^{n - 1}$.
\end{theorem*}
The first main result of this paper is to extend this theorem to a larger family of reflection groups over finite fields, namely, to the families of linear and unitary groups over $\Fq$ that contain a Singer cycle.  
\begin{main}\label{main:short}
    Let $G$ be one of the groups for which $\SLnFq \leq G \leq \GLnFq$ or $\SU_n(q) \leq G \leq \GU_n(q)$ such that $G$ contains a Singer cycle $c$.  Then the number of factorizations of $c$ as a product of the minimum number $n$ of reflections in $G$ is $h^{n - 1}$, where $h = |c|$ is the order of $c$.
\end{main}

In the representation theory and combinatorics of the finite unitary group, the phenomenon of \emph{Ennola duality} occurs when the answer to a question about $\GU_n(q)$ can be obtained by substitution of `$-q$' for `$q$' into the answer to the corresponding question for $\GLnFq$.  Main Theorem~\ref{main:short} is an instance of this phenomenon.
\begin{main}\label{main:ennola}
Fix a positive integer $n$ for which both $\GLnFq$ and $\GU_n(q)$ contain a Singer cycle.  There exists a polynomial $P = P_{n}$ with the following property: for each prime power $q$, the number of factorizations of the Singer cycle in $\GLnFq$ as a product of $n$ reflections is $P(q)$, and the number of factorizations of the Singer cycle in $\GU_n(q)$ as a product of $n$ reflections is $P(-q)$.
The same statement is true if we replace $\GLnFq$ and $\GU_n(q)$ respectively with $\SLnFq$ and $\SU_n(q)$.
\end{main}

As in \cite{LRS}, we proceed by algebraic means, using a standard technique to reduce the enumerative problem to a calculation of sufficiently many character values of the groups under consideration.  This technique has the virtue that it allows one to enumerate factorizations of a Singer cycle in these groups into any number of reflections.  The statement of these formulae uses the following standard (``$q$-analogue'') notations: for nonnegative integers $k \leq n$ and $q$ a free parameter, let
\begin{align*}
    [n]_q & := 1 + q + q^2 + \ldots + q^{n - 1}, \\
    [n]!_q & := [n]_q \cdot [n - 1]_q \cdots [2]_q \cdot [1]_q, && \text{and}\\
    \qbin{n}{k}{q} & := \frac{[n]!_q}{[k]!_q \cdot [n - k]!_q}.
\end{align*}
\begin{main}\label{main:longer factorizations}
Let $G$ be one of the groups for which $\SLnFq \leq G \leq \GLnFq$ or $\SU_n(q) \leq G \leq \GU_n(q)$ such that $G$ contains a Singer cycle $c$, and let $X = \{\det(g) \colon g \in G\}$.  Then for $\ell \geq n$, the number of factorizations of $c$ as a product of $\ell$ reflections is
\[
[n]_{\vep q}^{\ell - 1} \sum_{i = 0}^{\ell - n} (-1)^i \binom{\ell}{i} \qbin{\ell - i - 1}{n - 1}{\vep q} |X|^{\ell - i - 1},
\]
where $\vep = 1$ if $G$ is linear and $\vep = -1$ if $G$ is unitary.
\end{main}

Our proofs of these results proceed in two stages: first, we use the Deligne--Lusztig theory of representations of algebraic groups to establish a series of Ennola dualities for relevant character sums, thereby proving Theorem~\ref{main:ennola} (as Theorem~\ref{thm:SLU Ennola} in the case of the special groups, and as Theorem~\ref{thm:|X| > 1} for the others).  Then afterwards, we make use of the earlier, more elementary theory of characters of $\GLnFq$ in order to carry out the explicit computations necessary to prove Theorems~\ref{main:short} and~\ref{main:longer factorizations} (as Theorem~\ref{thm:SL-enumeration-arbitrary-length} in the case of the special groups, and as Corollary~\ref{cor:counting |X| > 1} for the others).  Background for all the needed machinery is provided in Section~\ref{sec:background}, while the proofs of the main results are the subject of Section~\ref{sec:proofs}.  We end in Section~\ref{sec:final remarks} with a number of open questions and other remarks.

\section{Background}\label{sec:background}

\subsection{Finite fields}

Let $q$ be a prime power and fix (once and for all) a finite field $\Fq$ and its algebraic closure $\KK := \ol{\Fq}$.  For each positive integer $d$, there is a unique extension of $\Fq$ of degree $d$ inside $\KK$, which we denote by $\FF_{q^d}$.  We recover the field $\Fq$ as the fixed points $\Fq = \KK^{F_+}$ in $\KK$ of the \defn{standard Frobenius map} $F_+: \KK \to \KK$ defined by $F_+(a) = a^q$, and more generally $\FF_{q^d} = \KK^{F_+^d}$ consists of the fixed points of the $d$th iterate $F_+^d$.  The multiplicative subgroup $\Fq^\times$ of $\Fq$ (respectively, $\FF_{q^d}^\times$ of $\FF_{q^d}$) is cyclic, of order $q - 1$ (resp., $q^d - 1$).

In what follows, we will also consider the \defn{twisted Frobenius map} $F_- : \KK \to \KK$ defined by $F_-(a) = a^{-q}$.  When $d$ is even, we have the equality $\KK^{F_-^d} = \KK^{F_+^d} = \FF_{q^d}$, but when $d$ is odd we have that the fixed points $\KK^{F_-^d}$ of $F_-^d$ in $\KK$ consist of $0$ together with a cyclic group of order $q^d + 1$, consisting of the $(q^d - 1)$th powers in $\FF_{q^{2d}}^\times$.

\subsection{Reflection groups}
\label{sec:reflection groups}

Given a finite-dimensional vector space $V$, a \defn{reflection} in $\GL(V)$ is any non-identity linear transformation that fixes a hyperplane pointwise.\footnote{In the literature---e.g., in \cite{KemperMalle}---these elements are sometimes known as \emph{pseudoreflections}, with the unmodified noun restricted to diagonalizable reflections with determinant $-1$.  This subclass of reflections does not play any special role in the present work, so we opt for the simpler language.  It is also common to require reflections to have finite order; since we will only consider finite groups, this is unnecessary.}  If $t$ is a reflection such that $\det(t) = \alpha \neq 1$, then $t$ is diagonalizable, with matrix \[ \begin{bmatrix} \alpha & & & \\ & 1 & & \\ & & \ddots & \\ & & & 1\end{bmatrix}\] in an appropriate basis of eigenvectors.  If $\det(t) = 1$, then we say that $t$ is a \defn{transvection}; its matrix in an appropriate basis is the Jordan block \[ \begin{bmatrix} 1 & 1 & & & \\ & 1 & & & \\ & & 1 & & \\ & & & \ddots & \\ & & & & 1\end{bmatrix}.\]
Say that a subgroup $G$ of $\GL(V)$ is a \defn{reflection group} if $G$ is generated by reflections.  

Over fields of characteristic $0$, the finite reflection groups are fully classified (see, e.g., \cite[Ch.~15]{Kane}).  This is not the case over fields of positive characteristic.\footnote{The classification project for reflection groups over positive characteristic breaks down at the first step: in characteristic $0$, Maschke's theorem implies that each reducible reflection group decomposes as a direct product; however, the proof of Maschke's theorem involves division by the order of the group, so fails in positive characteristic.  Indeed it is easy to find reducible reflection groups in positive characteristic that do not decompose: for example, the cyclic group of matrices over $\Fq$ generated by a single transvection is reducible (it stabilizes the fixed space of the reflection) but it is not a direct product.  The \emph{irreducible} reflection groups over finite fields \emph{have} been fully classified: the (complicated!) classification is the result of work of numerous authors \cite{Kantor-1979, Wagner1, Wagner2, ZS}, and a clear description of it is given in \cite{KemperMalle}.}
Nevertheless, many important families of groups belong to this class, including the following \cite[\S1]{KemperMalle}:
\begin{itemize}
\item the \defn{general linear group} $\GLnFq$ of all $n \times n$ invertible matrices over $\Fq$;
\item its subgroup $\SLnFq$ (the \defn{special linear group}) of invertible matrices of determinant $1$;
\item the \defn{general unitary group} $\GU_n(q) = \{g = (g_{i, j}) \in \GL_n(q^2) : g^{-1} = (g_{j, i}^q)\}$;
\item its subgroup $\SU_n(q) = \{ g \in \GU_n(q) \colon \det(g) = 1\}$ (the \defn{special unitary group}), except in the case $n = 3$ and $q = 2$; and
\item the different flavors of orthogonal and symplectic groups, which preserve a symmetric or alternating form on $V$.
\end{itemize}
Various of these groups belong to the families of \emph{classical groups} or \emph{finite groups of Lie type}.  In what follows, we will be particularly interested in the linear and unitary groups, for which we now introduce some specialized notation, following \cite{SF-V}.

\subsection{Linear and unitary groups}

Fix a positive integer $n$ and a prime power $q$.  For an $n$-dimensional vector space $V$ over $\Fq$, the general linear group $\GL(V)$ of invertible linear transformations on $V$ is isomorphic (by choosing a basis) to the group $\GL^+_n := \GLnFq$.  We make the identification $\GL^+_1 = \Fq^\times$ for the group of determinants of elements of $\GL^+_n$.  The subgroup consisting of matrices with determinant $1$ is the special linear group $\SL^+_n := \SLnFq$.  

Similarly, if $V$ is a vector space over $\FF_{q^2}$ and $\langle \cdot , \cdot \rangle$ is a nondegenerate Hermitian form on $V$, the group of linear transformations on $V$ that preserve $\langle \cdot , \cdot  \rangle$ is isomorphic to the general unitary group $\GL^-_n := \GU_n(q)$.  We make the identification $\GL^-_1 = \{a^{q - 1} : a \in \FF_{q^2}^\times\} = \{ a \in \FF_{q^2} : a^{q + 1} = 1\}$ for the group of determinants of elements of $\GL^-_n$.  The subgroup of $\GL^-_n$ of matrices with determinant $1$ is the special unitary group $\SL^-_n := \SU_n(q)$.

We will denote by $\vep$ an arbitrary sign, i.e., $\vep \in \{+, -\}$, so that $\GLU_n$ denotes either the general linear or general unitary group (depending on $\vep$).  In formulae, we will identify the sign with the associated integer $\pm 1$.  For example, we have that the cardinality of $\GLU_n$ is given by the simple attractive formula
\begin{equation}\label{eq:GLU cardinality}
    |\GLU_n| = q^{\binom{n}{2}} \prod_{i = 1}^n (q^i - \vep^i) = \vep^n \cdot (\vep q)^{\binom{n}{2}} \cdot (\vep q - 1)^n \cdot [n]!_{\vep q},
\end{equation}
where the factors in the first product should be understood as $q^i - 1$ when $\vep = +$ and $q^i - (-1)^i$ when $\vep = -$.

The following result is neither original nor stated in maximal generality; we include it because it is convenient to have an explicit description of the groups $G$ such that $\SLU_n \leq G \leq \GLU_n$.

\begin{propdef}
\label{prop:between SL and GL}
Fix a positive integer $n$, prime power $q$, and sign $\vep$.  For any subgroup $X$ of $\GLU_1$, let $G_X := \{ g \in \GLU_n \colon \det(g) \in X\}$.  The following hold.
\begin{enumerate}[(a)]
\item Except in the case $\vep = -$, $n = 3$, $q = 2$, $X = \{1\}$, the group $G_X$ is a reflection group.
\item Suppose that $G$ is a group such that $\SLU_n \leq G \leq \GLU_n$, and let $X := \{ \det(g) \colon g \in G\}$.  Then $G = G_X$.  In particular, $G$ is a reflection group except in the case $G = \SU_3(2)$.
\end{enumerate}
\end{propdef}
\begin{proof}
(a) If $n = 1$, the group $\GLU_n = \GLU_1$ is cyclic, and all non-identity elements are reflections, so the result is trivial.  (This includes the case $\SLU_1 = \{1\}$, generated by the empty set of reflections.)

If $\vep = -$, $n = 3$, $q = 2$, then $\SLU_n = \SU_3(2)$ is \emph{not} generated by reflections (the transvections generate a subgroup of index $4$), but $\GLU_n = \GU_3(2)$ \emph{is} generated by reflections (see \cite[Thms.~1.4 and~1.5]{KemperMalle}).  In this case $\GLU_1$ is cyclic of order $q + 1 = 3$.  Thus its unique proper subgroup is the trivial group, and there are no intermediate subgroups $X$ to consider.

Now suppose that $n \geq 2$ and we are not in the previous case.  Consider an arbitrary element $g \in G_X$.  If $\det(g) = 1$ then $g \in \SLU_n$ and so $g$ is a product of transvections \cite[Thms.~1.5~\&~8.1]{KemperMalle}.  Otherwise, $\det(g) \neq 1$, and so the matrix
\[
t = \bmat{\det(g) &&& \\ &1&& \\ &&\ddots& \\ &&&1}
\]
is a semisimple reflection in $\GLU_n$ such that $\det(t) = \det(g)$.  Thus $t \in G_X$ and $g = t \cdot u$ where $u \in \SLU_n$. Hence in this case $g$ is a product of the reflection $t$ together with some transvections.  In either case, $g$ is a product of reflections in $G_X$.

(b) Choose such a subgroup $G$, and let $X := \{ \det(g) \colon g \in G\}$.  Since the determinant is multiplicative and all elements of $G$ are invertible, $X$ is a subgroup of $\GLU_1$.  By definition of $X$, it follows immediately that $G \subseteq G_X$, where $G_X$ is as in (a).   On the other hand, let $g$ be any element of $G_X$ and let $h$ be an element of $G$ such that $\det(h) = \det(g)$.  Then $u := g h^{-1} \in \SLU_n \subseteq G$, and consequently $g = uh \in G$.  Thus $G_X \subseteq G$, and so in fact $G = G_X$, as claimed.
\end{proof}

\subsection{Irreducible transformations}

Let $V$ be a finite-dimensional vector space over $\Fq$.  Say that a linear transformation $g$ on $V$ is \defn{irreducible} (or \defn{regular elliptic}) if $g$ does not stabilize any nontrivial subspace of $V$.  This property can be characterized in many equivalent ways.  One of these involves the following natural embedding $\FF_{q^n}^\times \hookrightarrow \GLnFq$.

The field $\FF_{q^n}$ is an $n$-dimensional vector space over $\Fq$.  For each element $a$ of $\FF_{q^n}$, the multiplication map $x \mapsto ax$ is $\Fq$-linear.  This induces an embedding $\FF_{q^n}^\times \hookrightarrow \GL_{\Fq}(\FF_{q^n})$ from the multiplicative subgroup of $\FF_{q^n}$ into the group $\GL_{\Fq}(\FF_{q^n})$ of $\Fq$-linear transformations of $\FF_{q^n}$.  The choice of an $\Fq$-basis for $\FF_{q^n}$ induces an isomorphism $\GL_{\Fq}(\FF_{q^n}) \cong \GLnFq$, and composing these gives a map
\begin{equation}
\label{eq:embedding}
\FF_{q^n}^\times \hookrightarrow \GLnFq
\end{equation}
that sends the element $a$ to the matrix of its multiplication map.

\begin{proposition}[{\cite[Prop.~4.4]{LRS}}]
\label{regular-elliptic-definition-proposition}
The following are equivalent for $g$ in $\GLnFq$.
\begin{enumerate}[(i)]
\item There are no nonzero proper $g$-stable $\Fq$-subspaces inside $\Fq^n$.
\item The characteristic 
polynomial $\det(zI_n-g)$ is irreducible in $\Fq[z]$.
\item The element $g$ is the image of some $\beta$ in $\FF_{q^n}^\times$ satisfying
$\Fq(\beta)=\FF_{q^n}$ (that is, a field generator for $\FF_{q^n}$ over $\Fq$) under one of the embeddings 
of \eqref{eq:embedding}.
\end{enumerate}
\end{proposition}

\begin{remark}\label{rem:rcf}
Concretely, given a field generator $\beta$ for $\FF_{q^n}$ over $\Fq$, one can write down one of its images under \eqref{eq:embedding} as follows: let $f(z) \in \Fq[z]$ be its minimal polynomial, necessarily of degree $n$ with roots $\beta, \beta^q, \beta^{q^2}, \ldots, \beta^{q^{n - 1}}$ (a complete $F_{+}$-orbit).  Expand
\[
f(z) = z^n - a_1 z^{n - 1} - \ldots - a_n.
\]
Then $(\beta, \beta^q, \ldots, \beta^{q^{n - 1}})$ is an ordered basis for $\FF_{q^n}$ over $\Fq$, and the image of $\beta$ under \eqref{eq:embedding} for this choice of basis is the \emph{companion matrix} \[
\bmat{
&&&& a_n \\
1 &&&& a_{n - 1} \\
& 1 &&& a_{n - 2} \\
&& \ddots && \vdots \\
&&& 1 & a_1}
\]
of $f$.
\end{remark}

\subsection{Singer cycles in the general linear group}\label{sec:GL Singer cycles}

Say that a linear transformation $c$ is a \defn{Singer cycle} for $\GLnFq$ if it is the image under one of the inclusions \eqref{eq:embedding} of a \emph{cyclic} generator for $\FF_{q^n}^\times$ (not merely a \emph{field} generator of $\FF_{q^n}$ over $\Fq$).  These elements may be characterized in many equivalent ways.
\begin{prop}[{see \cite[\S2.1]{Brookfield}, \cite[Lem.~3]{Gill}
}]
The following are equivalent for an element $c$ of $\GLnFq$.
\begin{enumerate}[(i)]
\item $c$ is a Singer cycle;
\item $c$ is irreducible of maximum possible multiplicative order;
\item $c$ has multiplicative order $q^n - 1$;
\item $c$ acts transitively on $\Fq^n \smallsetminus \{0\}$;
\item $c$ has as an eigenvalue one of the cyclic generators for $\FF_{q^n}^\times$.
\end{enumerate}
\end{prop}

It follows from their definitions that every irreducible element in $\GLnFq$ is a power of a Singer cycle.  Moreover, since $\FF_{q^n}$ is generated (as a ring) by any element that does not belong to one of the subfields $\FF_{q^d}$ for $d \mid n$, the power $c^k$ of a Singer cycle $c$ is irreducible as long as $\frac{q^n - 1}{k}$ does not divide $q^d - 1$ for any proper divisor $d$ of $n$.

In the literature, the term \emph{Singer cycle} is sometimes used to refer to a maximal cyclic group that acts irreducibly; we will call these groups \defn{Singer subgroups}.  Each Singer cycle (in our sense) generates such a subgroup.  While Singer \emph{cycles} are not all conjugate, the Singer \emph{subgroups} of $\GLnFq$ are all conjugate; this amounts to changing the choice of basis in \eqref{eq:embedding}.

If $c$ is a Singer cycle for $\GLnFq$ that is the image of $\alpha \in \FF_{q^n}^{\times}$ under the one of the inclusions \eqref{eq:embedding}, then the eigenvalues of $c$ over $\KK$ are precisely $\alpha, \alpha^q, \ldots, \alpha^{q^{n - 1}}$, and the determinant $\det(c) = \alpha^{(q^n - 1)/(q - 1)}$ is a cyclic generator for $\Fq^\times$.

\subsection{Singer cycles in other matrix groups}
\label{sec:Singer cycles in other groups}

Given a subgroup $G \leq \GLnFq$, say that an element $c$ in $G$ is a \defn{Singer cycle} for $G$ if $c$ is irreducible and has maximum multiplicative order among the irreducible elements in $G$.  Since every irreducible element in $\GLnFq$ is a power of a $\GLnFq$-Singer cycle, it follows that if $G$ is a subgroup of $\GLnFq$ that contains a ($G$-)Singer cycle $\wt{c}$, then there is a $\GLnFq$-Singer cycle $c$ such that $\wt{c} = c^k$, where $k$ is the smallest positive integer such that $c^k \in G$.

Not every reflection group over $\Fq$ contains irreducible elements, hence Singer cycles.  The following proposition identifies the cases of interest to us.
\begin{proposition}\label{prop:Singer cycles exist}
For all positive integers $n$ and all prime powers $q$, if $\SLnFq \leq G \leq \GLnFq$ then $G$ contains irreducible elements.

For all odd positive integers $n$ and all prime powers $q$, if $\SUnFq \leq G \leq \GUnFq$ then $G$ contains irreducible elements, with the unique exception of the group $\SU_3(2)$.  For $n$ even, $\GUnFq$ and its subgroups do not contain any irreducible elements.

Moreover, if $\SLU_n \leq G = G_X \leq \GLU_n$ (where $X$ and $G_X$ are as in Proposition/Definition~\ref{prop:between SL and GL}) and $G$ contains irreducible elements, then the order of any Singer cycle in $G$ is $|X| \cdot \frac{(\vep q)^n - 1}{\vep q - 1}$ and its determinant is a cyclic generator for $X$.
\end{proposition}
\begin{proof}
It was shown by Huppert \cite[Satz~4]{Huppert1970} that the unitary group $\GU_n(q)$ contains irreducible elements if and only if $n$ is odd.  In particular, no subgroup of $\GUnFq$ contains irreducible elements when $n$ is even.  Furthermore, Huppert showed that when $n$ is odd, the Singer cycles in $\GU_n(q)$ have order $q^n + 1$, i.e., they are the $(q^n - 1)$th powers of Singer cycles in $\GL_n(q^2)$.  This gives all the desired statements for $\GLU_n$. 

Consider next the special linear/unitary case.\footnote{This is treated by assertion in  \cite{Bereczky}; however, the discussion there contains an oversight in the case $\SU_3(2)$, where the intersection of any Singer subgroup for $\GL_3(2^2)$ with $\SU_3(2)$ consists precisely of the three scalar matrices.} When $n = 1$, we have that $\SLU_1$ is a trivial group, whose unique element (the identity) vacuously meets the definition of irreducible element (because there are no nontrivial subspaces to stabilize). We now consider $n > 1$.

Since the determinant of a $\GL_n(q^2)$-Singer cycle is a cyclic generator for $\GL_1(q^2)$, it follows by elementary number theory that the determinant of a $\GU_n(q)$-Singer cycle is a cyclic generator for $\GU_1(q)$.  Therefore, if $\SLU_n$ contains irreducible elements, then its Singer cycles are the $|\GLU_1|$th powers of the Singer cycles in $\GLU_n$.  

In the linear ($\vep = +$) case for $n \geq 2$, we have $\frac{q^n - 1}{q - 1} > q^{n - 1} - 1$, and so $\frac{q^n - 1}{q - 1}$ is not a divisor of $q^d - 1$ for any $d < n$.  Consequently, the $(q - 1)$th power of a $\GLnFq$-Singer cycle is irreducible, and so is an $\SLnFq$-Singer cycle.

In the unitary ($\vep = -$) case for $n \geq 3$ odd, we consider when the $(q + 1)$th power of a $\GUnFq$-Singer cycle could fail to be irreducible.  So suppose that $\frac{q^n + 1}{q + 1} \mid q^d - 1$ for $d$ a proper divisor of $2n$.  By multiplying out and using the fact that $q \geq 2$, we have that $\frac{q^n + 1}{q + 1} > q^{n - 2} - 1$.  Thus, under our hypothesis, either $d = n$ (which always divides $2n$) or $d = n - 1$ and $n - 1 \mid 2n$.  We have by elementary algebra that
\[
q \cdot \frac{q^n + 1}{q + 1} < q^n - 1 < (q + 1)\cdot \frac{q^n + 1}{q + 1}
\]
for all $n \geq 3$ and $q \geq 2$, so $q^n - 1$ lies between two consecutive multiples of $\tfrac{q^n + 1}{q + 1}$, and the first case does not occur.  In the second case, the condition $n - 1 \mid 2n$ is equivalent to $n - 1 \mid 2$, so (since $n$ is odd) that $n = 3$.  When $n = 3$, we have
\[
1 \cdot \frac{q^3 + 1}{q + 1} \leq 
q^2 - 1 = \frac{q^3 + 1}{q + 1} + (q - 2) < 2 \cdot \frac{q^3 + 1}{q + 1}.
\]
Thus, this case occurs if and only if $n = 3$ and $q = 2$.  Combining the considerations above, we have that for $n \geq 3$ odd and any $q$, the $(q + 1)$th power of a $\GUnFq$-Singer cycle is irreducible, and so is a $\SUnFq$-Singer cycle, except when $(n, q) = (3, 2)$, and that $\SU_3(2)$ does not contain a Singer cycle.

We now consider the intermediate groups $G$ such that $\SLU_n < G < \GLU_n$.  As noted in the proof of Proposition/Definition~\ref{prop:between SL and GL}, in the exceptional case $\vep = -$, $n = 3$, $q = 2$, there are no intermediate groups to consider.  In all other cases, the preceding paragraphs establish that if $n, q, \vep$ are such that $\GLU_n$ contains irreducible elements, then $\SLU_n$ does as well.  Therefore, $G$ contains irreducible elements.  Let $X = \{\det(g) \colon g \in G\}$, so that $G = G_X$ by Proposition/Definition~\ref{prop:between SL and GL}.  Since $\GLU_1$ is cyclic, setting $x = |X|$ and $d = \frac{q - \vep}{x}$, we have $X = \{a \in \GLU_1 \colon a^x = 1\} = \{a^d \colon a \in \GLU_1\}$.  Therefore, the Singer cycles for $G_X$ are all equal to $c^d$ for some $\GLU_n$-Singer cycle $c$, and consequently have multiplicative order $\frac{|\GLU_1|}{d} = x \cdot \frac{(\vep q)^n - 1}{\vep q - 1}$ and determinant $\det(c)^{d}$, a cyclic generator for $X$.  
\end{proof}

\begin{remark}\label{rmk:Singer eigenvalues}
    For $n$ odd, let $c'$ be a Singer cycle in $\GU_n(q)$, and let $c$ be the Singer cycle in $\GL_n(q^2)$ such that $c' = c^{q^n - 1}$.  From the description of the eigenvalues of $c$ in Section~\ref{sec:GL Singer cycles} above, we have that the eigenvalues of $c'$ are $\alpha^{q^n - 1}, \alpha^{q^2(q^n - 1)}, \ldots, \alpha^{q^{2(n - 1)}(q^n - 1)}$ where $\alpha$ is a generator for $\GL_1(q^{2n}) = \FF_{q^{2n}}^\times$.  In this case, $\gamma := \alpha^{q^n - 1}$ is a cyclic generator for $\GU_1(q^n)$, and so (using that $\gamma^{q^{2k}} = \gamma^{(-q)^{2k}}$ for $k < n/2$ and $\gamma^{q^{2k}} = \gamma^{- q^{2k - n}} = \gamma^{(-q)^{2k - n}}$ for $k > n/2$) the eigenvalues of $c'$ are $\gamma, \gamma^{-q}, \gamma^{(-q)^2}, \ldots, \gamma^{(-q)^{n - 1}}$, a complete $F_{-}$-orbit.
\end{remark}

\subsection{Enumeration: the character approach}

In any reflection group, the set of reflections is closed under conjugacy.  There is a standard technique, dating to the 19th century, for enumerating factorizations in a group when the allowable factors have this property.  For a finite group $G$, let $\Irr(G)$ denote the set of its irreducible finite-dimensional complex representations, and for a representation $V$ of $G$ let $\chi^V$ denote the associated character.
\begin{proposition}[Frobenius \cite{frobenius}]
\label{general factorization prop}
Let $G$ be a finite group, and $A_1,\ldots,A_\ell \subseteq G$
unions of conjugacy classes in $G$.  Then for $g$ in $G$,
the number of ordered factorizations $(t_1,\ldots,t_\ell)$ with
$g=t_1 \cdots t_\ell$ and $t_i$ in $A_i$ for $i=1,2,\ldots,\ell$ is
\begin{equation}
\label{frobenius factorization equation}
\frac{1}{|G|} \sum_{V \in \Irr(G)} 
 \deg(V)^{1 - \ell} \chi^V(g^{-1}) \cdot \prod_{i = 1}^\ell \sum_{t \in A_i} \chi^V(t).
\end{equation}
\end{proposition}
Proposition~\ref{general factorization prop} reduces the enumeration problem we consider to the problem of understanding the character table of our groups sufficiently well.  Moreover, for $G$ such that $\SLU_n \leq G \leq \GLU_n$, it follows from Proposition/Definition~\ref{prop:between SL and GL} that the set of $G$-reflections is actually closed under conjugation \emph{in the larger group $\GLU_n$}; thus it suffices to understand just the character theory of $\GLU_n$.  The next sections are devoted to describing this theory, which is based on the connection between the finite groups $\GLU_n$ and the algebraic group $\GL_n(\KK)$ to which they belong.

\subsection{Levi subgroups, tori, and conjugacy classes}\label{sec:conjugacy classes}

The definitions of the general linear and unitary groups in Section~\ref{sec:reflection groups} may be rephrased in the language of \emph{algebraic groups}: $\GLnFq$ and $\GUnFq$ are respectively the fixed points (or \defn{rational points}) in $\GG := \GL_n(\KK)$ of the entrywise \defn{Frobenius map} $(g_{i, j}) \mapsto (g_{i, j}^q)$ and the \defn{twisted Frobenius map}\footnote{Such maps are called \emph{Steinberg maps} in some sources, e.g., \cite{MalleTesterman}.} $(g_{i, j}) \mapsto (g_{j, i}^q)^{-1}$ (see \cite[1.17]{Carter}).  In an abuse of notation, we denote these maps respectively by $F_+$, $F_-$.

Every algebraic group has a special family of subgroups called \emph{Levi subgroups}.  In the present work, we will not need a general definition of Levi subgroups in the context of algebraic groups, which may be found in \cite[Ch.~1]{DigneMichel}.  Instead, we give a concrete description of the (rational) Levi subgroups of $\GLU_n$.  Up to conjugacy, these are indexed by multisets $\{(d_1, m_1), \ldots, (d_{k}, m_k)\}$ of pairs of positive integers such that $d_1m_1 + \ldots + d_k m_k = n$, where one representative of this multiset is a block-diagonal subgroup of $\GLU_n$ isomorphic to
\[
\GL_{m_1}^{\vep^{d_1}}(q^{d_1}) \times \cdots \times \GL_{m_k}^{\vep^{d_k}}(q^{d_k}).
\]
Here the $i$th block (of size $d_im_i \times d_im_i$) comes from embedding $\GL_{m_i}^{\vep^{d_i}}(q^{d_i})$ into $\GLU_{m_id_i}$ via a map generalizing \eqref{eq:embedding} (which can be done while respecting the necessary symmetry in the unitary case).  In other words, in the case of $\GLnFq$, each factor is of the form $\GL_m(q^d)$, while in the case of $\GUnFq$, the factors are either of the form $\GU_m(q^d)$ (if $d$ is odd) or $\GL_m(q^d)$ (if $d$ is even).  

A \defn{maximal torus} in $\GG$ is any subgroup isomorphic to $(\KK^\times)^n$.  Given a $\Frob$-stable maximal torus in $\GG$, its $\Frob$-fixed points form a \defn{(rational) maximal torus} in $\GLU_n$.  The maximal tori of $\GG$ are all conjugate in $\GG$; however, the different maximal tori interact in different ways with the Frobenius maps, leading to multiple conjugacy classes of rational maximal tori in $\GLU_n$.  These are indexed as follows.  A \defn{partition} $\lambda$ is a finite nonincreasing sequence of positive integers. We write $\lambda \vdash n$ or $|\lambda| = n$ to denote that the sum of the parts of $\lambda$ is $n$ (that ``$\lambda$ is a partition of $n$'').  For each partition $\lambda = \langle \lambda_1, \ldots, \lambda_k\rangle$ of $n$, choose a Levi subgroup $\TTT_\lambda$ indexed by $\{(\lambda_1, 1), \ldots, (\lambda_k, 1)\}$.  Then this subgroup is a rational maximal torus of $\GLU_n$, and furthermore every rational maximal torus of $\GLU_n$ is conjugate in $\GLU_n$ to $\TTT_\lambda$ for exactly one partition $\lambda$ of $n$.  In $\GLnFq$ for all $n$ and $\GUnFq$ for $n$ odd, the Singer subgroups are precisely the conjugates of $\TTT_{\langle n\rangle}$.

We are now ready to describe the conjugacy classes of arbitrary elements of $\GLU_n$.  Let $\rootorbs$ denote the set of $\Frob$-orbits in $\KK^\times$.  As in Remark~\ref{rem:rcf}, we identify orbits in $\rootorbs$ with polynomials via
\[
\phi \in \rootorbs \quad \longleftrightarrow \quad  \prod_{\alpha \in \phi} (z - \alpha).
\]
Let $\Par$ denote the set of all integer partitions, and $\Par_n$ the partitions of $n$.  The conjugacy classes in $\GLU_n$ are indexed by \defn{multipartitions}, that is, by functions $\llambda: \rootorbs \to \Par$ such that
\begin{equation}\label{eq:cc indexing}
n = \sum_{\phi \in \rootorbs} |\phi| \cdot |\llambda(\phi)|,
\end{equation}
as we describe now.

Say that an element of $\GG$ (in particular, of $\GLU_n$) is \defn{semisimple} if it is diagonalizable over the (algebraically closed) field $\KK$ (equivalently, if it belongs to some maximal torus of $\GG$, or if it has multiplicative order relatively prime to $q$), and that an element is \defn{unipotent} if its eigenvalues over $\KK$ are all equal to $1$ (equivalently, if its multiplicative order is a power of the characteristic).  In this context, the \defn{Jordan canonical form} asserts that every element $g$ of $\GG$ can be written uniquely as $g = su$ where $s$ is semisimple, $u$ is unipotent, and $s$ and $u$ commute \cite[1.4]{Carter}.  This decomposition descends to the finite groups $\GLU_n$, i.e., if $g \in \GLU_n$ then $u, s \in \GLU_n$ as well, by uniqueness.

In each conjugacy class of $\GLU_n$, there is a natural choice of representative for which the Jordan decomposition can be made completely explicit, extending the construction of Remark~\ref{rem:rcf} to use block-diagonal matrices whose blocks are companion matrices of polynomials, the identity, or $0$---see \cite[\S3]{ThiemVinroot} for details.  For our purposes, it suffices to remark that each semisimple element $s$ in $\GLU_n$ encodes a collection of orbits in $\rootorbs$ (namely, those containing its eigenvalues), that the unipotent part of the decomposition $g = su = us$ encodes simultaneously for each of these orbits $\phi$ a partition $\llambda(\phi)$, and that this association automatically satisfies \eqref{eq:cc indexing}.  In particular, representing a multipartition $\llambda$ by the set of pairs $(\phi, \llambda(\phi))$ for which $\llambda(\phi)$ is not the empty partition, we have the following special cases:
\begin{itemize}
    \item if $g$ is itself unipotent, then in its Jordan decomposition we have $s = 1$ and $u = g$, and the multipartition indexing its conjugacy class is $\llambda = \big\{ (z - 1, \lambda) \big\}$, where $\lambda$ is the ``usual'' partition that arises in elementary linear algebra from the theory of generalized eigenvectors;
    \item specifically, the conjugacy class of transvections in $\GLU_n$ is indexed by the multipartition $\big\{(z - 1, \langle 2, 1^{n - 2}\rangle)\big\}$;
    \item if $g$ is itself semisimple, then in its Jordan decomposition we have $s = g$ and $u = 1$, and the multipartition indexing its conjugacy class is given by $\llambda = \big\{ (f_1, \langle 1^{k_1}\rangle), (f_2, \langle 1^{k_2}\rangle), \ldots \big\}$ where each $f_i$ represents a single $\Frob$-orbit and the characteristic polynomial of $g$ factors as $f_1^{k_1} \cdot f_2^{k_2} \cdots$;
    \item specifically, the conjugacy class of semisimple reflections of determinant $\alpha$ is indexed by the multipartition $\big\{ (z - 1, \langle 1^{n - 1}\rangle ), (z - \alpha, \langle 1\rangle)\big\}$; and
    \item if $c$ is an irreducible element with characteristic polynomial $f$ (necessarily irreducible), its conjugacy class is indexed by the multipartition $\big\{(f, \langle 1\rangle)\big\}$.
\end{itemize}
We call multipartitions that have support of size $1$ (like those indexing the transvections and irreducible elements) \defn{primary}.


\subsection{Centralizers}\label{sec:centralizers}

In what follows, we will need at several points to write down explicitly the centralizers of certain elements in $\GLU_n$, or their cardinalities.  These have been worked out in \cite{Green, Ennola, Wall}.  We mention the cases that will be of interest to us in detail.

For a semisimple element $w$ in $\GLU_n$, the centralizer $C_{\GLU_n}(w)$ consists of those elements that stabilize the eigenspaces of $w$.  In the case that $w = t$ is a semisimple reflection, this means that the centralizer $C_{\GLU_n}(t)$ is conjugate to the block-diagonal direct product $\GLU_1 \times \GLU_{n - 1}$.  Therefore, by \eqref{eq:GLU cardinality}, the number of semisimple reflections in $\GLU_n$ with a given determinant $\alpha \neq 1$ is 
\[
\frac{|\GLU_n|}{|\GLU_1| \cdot |\GLU_{n - 1}|} = \frac{(\vep q)^{n - 1} \cdot ((\vep q)^n - 1)}{\vep q - 1}.
\]

The centralizer $C_{\GLnFq}(c)$ of a Singer cycle $c$ in $\GLnFq$ is the subgroup $\langle c \rangle$ generated by $c$ \cite[Satz~II.7.3]{Huppert1967}.  For any irreducible element $x$ in $\GLnFq$, choose a Singer cycle $c$ and an integer $k$ such that $x = c^k$.  Since $x$ is irreducible, it has distinct eigenvalues over $\KK = \ol{\Fq}$.  Therefore, any $g \in \GLnFq$ that commutes with $x$ must stabilize its $n$ eigenspaces.  But since $x = c^k$, $x$ and $c$ have the same eigenspaces, and so $g$ commutes with $c$, as well.  Therefore the centralizer of $x$ in $\GLnFq$ is also the Singer subgroup $\langle c \rangle$.  It follows immediately that the same holds for irreducible elements $x$ in any subgroup of $\GLnFq$, and consequently in particular for Singer cycles in $\GUnFq$ (where now the previous $q$ must be taken to be $q^2$).

\begin{remark}\label{rem:connected}
In the general linear and general unitary groups (but not in all finite groups of Lie type), the centralizers of semisimple elements are always Levi subgroups.  In the general character theory introduced in the next subsection, it will also be necessary to consider the $\Frob$-fixed points of the connected component $C^\circ_{\GG}(\cdot)$ (in the Zariski topology) that contains the identity of a centralizer.  However, because $[\GG, \GG] = \SL_n(\mathbb{K})$ is simply connected, it follows from \cite[Thm.~14.16]{MalleTesterman} that these centralizers are connected, so that $C^\circ_{\GG}(\cdot)^{\Frob} = C_{\GG}(\cdot)^{\Frob} = C_{\GLU_n}(\cdot)$ in all cases we will consider.  We note further that all Levi subgroups of $\GG$ are connected, being products of linear groups, as are the centralizers of semisimple elements in those Levi subgroups.
\end{remark}

The structure of the centralizer of a transvection is more complicated.  For our purposes, it is enough to note that the number of transvections in $\GLU_n$ is \cite[p.~37, 59]{Fulman-thesis}
\[\frac{((\vep q)^n - 1)((\vep q)^{n - 1} - 1)}{\vep q - 1}.\]

\subsection{Character indexing and the character formula}
\label{sec:characters}

For a positive integer $k$, let $T_{k} = \GLU_1(q^k)$.  For any positive integers $r, m$ with $r\mid m$, there is a norm map
\begin{align*}
N_{m,r}\colon T_m & \to T_r \\
x & \mapsto x^{\frac{(\vep q)^m - 1}{(\vep q)^r - 1}}
\end{align*}
and we may identify $\KK^\times$ with the inverse limit $\varprojlim T_m$ with respect to this system of maps.  As $T_m$ is abelian, its irreducible (complex-valued) characters form its \defn{character group $\wh{T}_m$} under pointwise multiplication.  For $r \mid m$, let $\wh{N}_{m,r} : \wh{T}_r \to \wh{T}_m$ denote the dual norm map, defined by $\wh{N}_{m,r}(\xi) = \xi\circ N_{m,r}$.  The direct limit $\varinjlim \wh{T}_r = \wh{\KK}^\times$ is the character group of $\KK^\times$, and the action of the Frobenius map $\Frob$ extends to this group via $\Frob(\zeta)(x) = \zeta(\Frob(x))$; let $\charorbs$ denote the set of $\Frob$-orbits in $\wh{\KK}^\times$.  The irreducible characters of $\GLU_n$ are indexed by functions $\llambda: \charorbs \to \Par$ such that
\[
n = \sum_{\phi \in \charorbs} |\phi| \cdot |\llambda(\phi)|,
\]
as we describe next (following \cite{SF-V, ThiemVinroot}).

The basic building-block for the character theory of the $\GLU_n$ is \defn{Lusztig induction}, of which we give a (very) partial account.  If $G$ is a finite group of Lie type to which Remark~\ref{rem:connected} applies (i.e., connectedness conditions hold), $L$ is a rational Levi subgroup of $G$, $\chi$ is an irreducible character of $L$, and $g$ is an element of $G$, the value $(R^G_L(\chi))(g)$ of the Lusztig-induced character $R^G_L(\chi)$ is given by
\begin{equation}
\label{eq:induction}
(R^G_L(\chi))(g) =  \frac{1}{|L| \cdot |C_{G}(s)|} \sum_{h \in G \colon s \in {^h L}} |C_{^hL}(s)| \sum_{v \in \left(C_{^hL}(s)\right)_\text{unip.}} \Grfn^{C_G(s)}_{C_{^hL}(s)}(u, v^{-1}) \cdot ({^h \chi})(sv)
\end{equation}
where $g = su$ is the Jordan decomposition of $g$, $^hL = hLh^{-1}$ denotes a conjugate of $L$ and $^h\chi$ the corresponding character of $^hL$ conjugate to $\chi$, $\left(C_{^hL}(s)\right)_\text{unip.}$ denotes the set of unipotent elements in $C_{^hL}(s)$, and $\Grfn$ is a \defn{(two-variable) Green function} \cite[Prop.~12.2(i)]{DigneMichel}.  These functions are defined \cite[Def.~12.1]{DigneMichel} in terms of traces on a certain $\ell$-adic cohomology; we omit the details because, as we see below, it will be possible to avoid using the definition directly.

Now restrict to the case that $G = \GLU_n$.  Choose an orbit $\phi \in \charorbs$ whose size $|\phi| = d$ divides $n$, and let $\gamma$ be a partition of $n/d$.  Each $\theta \in \phi$ is a character on $T_d$, hence also (via the norm map) on $T_{d \gamma_i}$ for each $i$.  By choosing a (fixed but arbitrary) identification  $\TTT_{d \gamma} := T_{d \gamma_1} \times  \cdots \times T_{d \gamma_k}$, $\theta$ further extends to a character of $\TTT_{d \gamma}$ by taking the product of the values on each factor.  The torus $\TTT_{d\gamma}$ is itself a Levi subgroup of $\GLU_n$, and the (Lusztig-)induced character $R_{\TTT_{d\gamma}}^{\GLU_n}(\theta)$ is called a \defn{Deligne--Lusztig character} of $\GLU_n$, first defined in \cite[1.20]{DeligneLusztig} (see also \cite[11.14]{DigneMichel}).  For any partition $\lambda \vdash n/d$, the formula
\begin{equation}
\label{eq:primary char}
\chi^{\phi, \lambda} := (-1)^{n + |\lambda|} \cdot\vep^{d \cdot a(\lambda') + \lfloor n/2 \rfloor} \cdot \sum_{\gamma \in \Par_{n/d}} \frac{\omega^\lambda(\gamma)}{z_\gamma} R^{\GLU_{n}}_{\TTT_{d\gamma}}(\theta)
\end{equation}
manifestly defines a class function on $\GLU_n$, where $\omega^\lambda$ is an irreducible character of the symmetric group $\Symm_{n/d}$ (indexed so that the one-row partition $(n)$ is the trivial character), $z_\gamma$ is the size of the centralizer in $\Symm_{n/d}$ of an element of cycle type $\gamma$, $a$ is the function\footnote{Often denoted ``$n$'' in the literature.} on partitions defined by $a(\mu) = \sum_i (i - 1)\mu_i$, $\lambda'$ denotes the transpose partition of $\lambda$, and $\theta \in \phi$ is any element of the orbit $\phi$.  In fact this class function is an irreducible character on $\GLU_n$: the sign in \eqref{eq:primary char} for the unitary ($\vep = -1$) case is given explicitly (and more generally) in the Remark after \cite[Thm.~4.3]{ThiemVinroot}, and for the linear ($\vep=1$) case this is exactly \cite[Lem.~7.5]{Green}. We call such characters \defn{primary irreducible characters} (as in \cite[\S8]{Green}). The degrees of the primary irreducible characters in $\GLU_n$ are given as follows.
For a nonnegative integer $d$ and an integer partition $\mu$, define the polynomial
\[
f^{d, \mu}(z) := \frac{z^{d \cdot a(\mu)}}{\prod_{y \in \mu} (z^{d \cdot h(y)} - 1)} \cdot \prod_{i = 1}^{d \cdot |\mu|} (z^{i} - 1) 
\]
where $h(y)$ denotes the hook-length of a cell $y$ in the Young diagram of $\mu$.  In the general linear group case ($\vep =1$), the following is from the work of J. A. Green \cite[Thm.~14]{Green}.  In the unitary case ($\vep = -1$), the result is stated explicitly in \cite[Thm.~5.1]{ThiemVinroot}, but also follows from more general results of Lusztig \cite[Thm.~4.23]{Lusztig}. 
\begin{proposition}
\label{prop:degrees}
For a prime power $q$, positive integers $d \mid n$, an orbit $\phi \in \charorbs$ of size $d$, and an integer partition $\lambda \vdash n/d$, the degree of the primary irreducible character $\chi^{\phi, \lambda}$ for $\GLU_n$ is
\[
\chi^{\phi, \lambda}(1) = \vep^{d \cdot a(\lambda) + d\cdot \sum_{y \in \lambda} h(y) + \binom{n + 1}{2}}  \cdot f^{d, \lambda}( \vep q).
\]
\end{proposition}
Since for any partition $\lambda$, $\sum_{y \in \lambda} h(y) = a(\lambda) + a(\lambda') + |\lambda|$, the sign in front may be written more simply as $\vep^{d \cdot a(\lambda') + \binom{n}{2}}$; also note that we have $(-1)^{\binom{n}{2}} = (-1)^{\lfloor n/2 \rfloor}$.

All other characters come from the primary characters by Lusztig induction: for any $\llambda: \charorbs \to \Par$ with $n = \sum_\phi |\phi| \cdot |\llambda(\phi)|$, the irreducible character of $\GLU_n$ indexed by $\llambda$ is given (for an appropriate choice of sign) by
\begin{equation}
\label{eq:all irreducibles}    
\chi^\llambda = \pm R^{\GLU_n}_L\left(\chi^{\phi_1, \llambda(\phi_1)} \times \cdots \times \chi^{\phi_k, \llambda(\phi_k)}\right)
\end{equation}
where $\{\phi_1, \ldots, \phi_k\}$ is the support of $\llambda$ (i.e., the set of orbits to which $\llambda$ assigns a nonzero partition) and $L$ is a Levi subgroup of $\GLU_n$ indexed by the multiset $\{ (|\phi_i|, |\llambda(\phi_i)|)\}$.

In the parameterization above, the trivial character of $\GLU_n$ is the primary character $\chi^{\llambda}$, where $\llambda(\phi) = \langle n\rangle$ and $\phi$ is the (singleton) orbit of the trivial character.  We note that we could alternatively parameterize the characters of $\GLU_n$ by replacing each partition by its transpose, and such a parameterization is also useful, see \cite[p.~721, Remark]{ThiemVinroot}.  One must of course take care in the corresponding formulae for characters when replacing $\lambda$ with $\lambda'$.

The preceding considerations essentially reduce the computations of character values of $\GLU_n$ to the evaluation of two-variable Green functions.  These evaluations are greatly simplified by the following considerations.  First, in every case in which we will want to compute a two-variable Green function evaluation $\Grfn^G_L(u, v)$, we will have that $L$ is a maximal torus $\TTT$ in $G$ and that $v = 1$ is the identity.  In this case, the two-variable Green function simplifies to a \emph{one-variable Green function} (as in, for example, \cite{ThiemVinroot}) by the relation $\Grfn_\TTT^G(u, 1) = \Grfn_\TTT^G(u)$ \cite[\S2.1]{Lubeck}.  Second, the Green function $\Grfn^{\GL_n}_\TTT(u)$ for the general linear group $\GLnFq$ is constructed to coincide exactly with a \defn{Green polynomial} $Q^{\lambda}_{\mu}(q)$ of symmetric function theory \cite[Ex.~2.4.20]{GeckMalle} \cite[Chs.~III.7,~IV.6]{Macdonald}, as follows: if $u$ is unipotent in $\GLnFq$ with associated partition $\mu$, then $\Grfn_{\TTT_\gamma}^{\GLnFq}(u) = Q^\mu_\gamma(q)$.  Moreover, Green functions for $\GLnFq$ expand nicely across direct products \cite[Lem.~2.5(a)]{GLT2020}, so that if $L = \GL_{m_1}(q^{d_1}) \times \cdots \times \GL_{m_k}(q^{d_{k}})$ is a Levi subgroup of $\GLnFq$ indexed by the multiset $\{(d_1, m_1), \ldots, (d_k, m_k)\}$, $u$ is a unipotent element of $L$, and $\TTT$ is a maximal torus of $L$ such that for each diagonal block $(d_i, m_i)$ of $L$, the unipotent $u$ assigns the partition $\mu(i) \vdash m_i$, while the portion of the torus $\TTT$ that lies in this block is conjugate to $\TTT_{\gamma(i)}$ for the partition $\gamma(i) \vdash m_i$, then
\[
\Grfn^L_{\TTT}(u) = \prod_{i = 1}^k Q^{\mu(i)}_{\gamma(i)}(q^{d_i}).
\]
Finally, as explained in \cite{ThiemVinroot}, the following theorem of Kawanaka (extending \cite[Thm.~3.1]{HottaSpringer}) tells us how to compute one-variable Green functions of the unitary group. 

\begin{theorem}[Ennola duality {\cite[Thm.~4.1.2]{Kawanaka}}; see also {\cite[Thm.~4.1]{ThiemVinroot}}]\label{thm:ennola}

Consider a Levi subgroup $L$ of $\GUnFq$, a unipotent element $u \in L$, and a maximal torus $\TTT$ of $L$.  Suppose that (as in Section~\ref{sec:conjugacy classes}) $L$ is indexed by the multiset $\{(d_1, m_1), \ldots, (d_k, m_k)\}$, and that for the block of $L$ with indices $(d_i, m_i)$, the unipotent $u$ assigns the partition $\mu(i) \vdash m_i$ while the portion of the torus that lies in this block is conjugate to $\TTT_{\gamma(i)}$ for a partition $\gamma(i) \vdash m_i$.  Then the Green function $\Grfn^{L}_{\TTT}(u)$ is given by 
\[
\Grfn^{L}_{\TTT}(u) = \prod_{i = 1}^k Q^{\mu(i)}_{\gamma(i)}\left( (-q)^{d_i} \right),
\]
where $Q^{\mu(i)}_{\gamma(i)}$ is a Green polynomial.
\end{theorem}

\section{Proofs of the main theorems}\label{sec:proofs}

The first step in the proof of our main theorems is to compute the character values necessary to apply Proposition~\ref{general factorization prop}.  We will see that for a Singer cycle $c$, $\chi(c) = 0$ for most irreducible characters; since this greatly reduces the number of character values we have to consider, we start there.

\subsection{Sparse character values on irreducible elements}

Fix an irreducible element $c$ in $\GLU_n$.  Since $c$ is semisimple, in its Jordan decomposition $c = su$ we have $s = c$ and $u = 1$.  Since $c$ is irreducible, by definition it does not stabilize any nontrivial subspace.  Thus, neither $c$ nor any of its conjugates lie in any proper subgroup of block-diagonal matrices.  It follows immediately from \eqref{eq:induction} that $\chi^\llambda(c) = 0$ when $\llambda$ is not primary, since there are no elements $h$ in $\GLU_n$ for which $c$ belongs to $^h L$ (equivalently, for which $h^{-1}ch$ belongs to $L$) when $L$ is a nontrivial direct product.

Furthermore, for $\chi^{\phi, \lambda}$ primary and $d = |\phi|$, if $\gamma \vdash n/d$ has more than one part then $\TTT_{d\gamma}$ is contained in a Levi subgroup of $\GLU_n$ of block-diagonal matrices with more than one block. Thus $(R^{\GLU_n}_{\TTT_{d\gamma}}(\theta))(c) = 0$ in this case.  Then it follows from \eqref{eq:primary char} that
\begin{equation}
\label{eq:primary hook Singer}
\chi^{\phi, \lambda}(c) = (-1)^{n + |\lambda|} \cdot\vep^{d \cdot a(\lambda') + \lfloor n/2 \rfloor} \cdot \frac{\omega^\lambda( \langle n/d \rangle )}{z_{\langle n/d \rangle}} (R^{\GLU_n}_{T_n}(\theta))(c).
\end{equation}
By the Murnaghan--Nakayama rule for character values in the symmetric group $\Symm_{n/d}$, $\omega^\lambda(\langle n/d \rangle) = 0$ unless $\lambda = \langle n/d - k, 1^k\rangle$ is \defn{hook-shaped} ($k = 0, \ldots, n/d - 1$).

The two preceding paragraphs amount to the proof of the following result, which extends \cite[Prop.~4.7(i,ii)]{LRS} to the unitary case.

\begin{proposition}\label{prop:mostly zero}
    If $c$ is an irreducible element in $\GLU_n$ (necessarily with $n$ odd in the case $\vep = -$) and $\chi = \chi^{\llambda}$ is an irreducible character for $\GLU_n$, then $\chi(c) = 0$ unless $\chi$ is a primary irreducible character, so that $\llambda$ is supported on a unique $\Frob$-orbit $\phi$ whose size $d$ is a divisor of $n$, and $\llambda(\phi)$ is a hook-shaped partition of $n/d$.
\end{proposition}

\begin{cor}
\label{cor:sparse}
Let $X$ be a multiplicative subgroup of $\GLU_1$.  Given an element $w$ of $G_X$, let $\Fac^X_\ell(w)$ denote the number of factorizations of $w$ as a product of $\ell$ reflections in $G_X$.  Let $c$ be a Singer cycle in $G_X$.  Then
\begin{multline*}
\Fac^X_\ell(c) = \frac{1}{|\GLU_n|} \sum_{d \mid n} \sum_{\substack{\phi \in \charorbs: \\ |\phi| = d}} \sum_{k = 0}^{n/d - 1} \Bigg( \deg(\chi^{\phi, \hook{k}{n/d}})^{1 - \ell} 
\times{}
\\
{}\times
\chi^{\phi, \hook{k}{n/d}}(c^{-1}) \cdot \left(\sum_{t \text{ refn.\ in } G_X}\chi^{\phi, \hook{k}{n/d}}(t)\right)^\ell \Bigg).
\end{multline*}
\end{cor}
\begin{proof}
Since $G_X$ is finite, a subspace in any linear representation of $G_X$ is stabilized by an element $g$ if and only if it is stabilized by $g^{-1}$.  Therefore, the inverse $c^{-1}$ of any irreducible element $c \in G_X$ (in particular, of any Singer cycle) is also irreducible.  Combining Proposition~\ref{general factorization prop} and Proposition~\ref{prop:mostly zero} (applied to the character values $\chi(c^{-1})$ on the irreducible element $c^{-1}$) immediately gives the desired result.
\end{proof}

\subsection{Some character sums for Singer cycles}

In the final computation, we will need the following sum of character values on Singer cycles.

\begin{prop}\label{prop:Singer sum}
    Let $G$ be one of the groups for which $\SLnFq \leq G \leq \GLnFq$ or $\SU_n(q) \leq G \leq \GU_n(q)$ such that $G$ contains a Singer cycle $c$.  For each positive integer $d \mid n$ and each hook-shaped partition $\hook{k}{n/d}$ of $n/d$, we have
    \[
    \sum_{\substack{\phi \in \charorbs: \\ |\phi| = d}} \chi^{\phi, \hook{k}{n/d}}(c)
    =
    \begin{cases}
    \pm Q^{\langle1\rangle}_{\langle1\rangle}((\vep q)^n) \cdot (q - \vep) \cdot \mu(d) & \text{ if } G = \SLU_n, \\
    0 & \text{ otherwise}.
\end{cases}
    \]
    In particular, in the case $G = \SLU_n$, the sign is given by $(-1)^{k + n + n/d} \cdot\vep^{d \cdot a(\lambda') + \lfloor n/2 \rfloor}$.
\end{prop}
\begin{proof}
We begin with the formula \eqref{eq:primary hook Singer} for the character value $\chi(c)$ at an irreducible element $c$ where $\chi := \chi^{\phi, \lambda}$ is a primary character, $|\phi| = d$ and $\lambda = \hook{k}{n/d}$ is a hook.  In this case, we have $z_{\langle n/d \rangle} = \frac{n}{d}$ (because the centralizer of a $m$-cycle in $\Symm_m$ is the cyclic subgroup it generates) 
and
$\omega^\lambda(\langle n/d \rangle) = (-1)^{k}$ (by the Murnaghan--Nakayama rule),
so 
\[
\chi^{\phi, \lambda}(c) = (-1)^{n + n/d} \cdot\vep^{d \cdot a(\lambda') + \lfloor n/2 \rfloor} \cdot
\frac{(-1)^k}{n/d} (R^{\GLU_n}_{T_n}(\theta))(c).
\]
Since $c$ is semisimple, in the Jordan decomposition $c = su$ for $c$ we have $s = c$ and $u = 1$.  Therefore, by \eqref{eq:induction}, we have
\begin{multline}
\label{eq:simplification 1}   
\chi^{\phi, \lambda}(c) = \frac{
(-1)^{k + n + n/d} \cdot\vep^{d \cdot a(\lambda') + \lfloor n/2 \rfloor} \cdot
d }
{n \cdot |T_n| \cdot |C_{\GLU_n}(c)| } 
\times{} 
\\
{}\times
\sum_{h \in \GLU_n \colon  c \in {^h T_n}} |C_{^hT_n}(c)| \sum_{v \in \left(C_{^hT_n}(s)\right)_\text{unip.}} \Grfn^{C_{\GLU_n}(c)}_{C_{^hT_n}(c)}(1, v^{-1}) \cdot ({^h \theta})(cv).
\end{multline}
Since $T_n$ (hence $^hT_n$) is abelian, the centralizer $C_{^h T_n}(c)$ is the whole torus ${^h T_n}$.  Moreover, $^hT_n$ consists entirely of semisimple elements, so $\left(C_{^hT_n}(s)\right)_\text{unip.}$ consists solely of the identity element. And it was also observed in Section~\ref{sec:centralizers} that the centralizer $C_{\GLU_n}(c)$ of $c$ in the larger group $\GLU_n$ is \emph{also} the whole torus.  Then \eqref{eq:simplification 1} can be simplified further to
\begin{equation}
\label{eq:simplification 2}   
\chi^{\phi, \lambda}(c) =  
\frac{
(-1)^{k + n + n/d} \cdot\vep^{d \cdot a(\lambda') + \lfloor n/2 \rfloor} \cdot
d }{n \cdot |T_n|^2 } \sum_{h \in \GLU_n \colon c \in {^h T_n}} |T_n| \cdot \Grfn^{^hT_n}_{^h T_n}(1, 1) \cdot ({^h \theta})(c).
\end{equation}
The values of $|T_n|$ and the Green function do not depend on $h$, so we can rewrite \eqref{eq:simplification 2} as
\begin{equation}
\label{eq:simplification 3}   
\chi^{\phi, \lambda}(c) =  
\frac{
(-1)^{k + n + n/d} \cdot\vep^{d \cdot a(\lambda') + \lfloor n/2 \rfloor} \cdot
d \cdot \Grfn^{T_n}_{T_n}(1, 1)}{n \cdot |T_n| } \sum_{h \in \GLU_n \colon c \in {^h T_n}} ({^h \theta})(c).
\end{equation}
Since $c$ has irreducible characteristic polynomial, there are $n$ conjugates of $c$ that belong to $T_n$, namely $c$, $c^{\vep q}$, \ldots, $c^{(\vep q)^{n - 1}}$.  Since $C_{\GLU_n}(c) = T_n$, the inner sum is equal to $|T_n| \sum_{i = 0}^{n - 1} \theta(c^{(\vep q)^i}) = |T_n| \sum_{i = 0}^{n - 1} \theta(\Frob^i(c))$.  Since furthermore $\phi$ is a Frobenius orbit of size $d$, we have $\sum_{i = 0}^{n - 1} \theta(\Frob^i(c)) = \frac{n}{d} \sum_{i = 0}^{d - 1} \theta(\Frob^i(c)) = \frac{n}{d} \sum_{\theta \in \phi} \theta(c)$, and so 
\begin{equation}
\label{eq:simplification 4}  
\chi^{\phi, \lambda}(c) 
=
(-1)^{k + n + n/d} \cdot\vep^{d \cdot a(\lambda') + \lfloor n/2 \rfloor} \cdot
\Grfn^{T_n}_{T_n}(1, 1) \cdot  \sum_{\theta \in \phi} \theta(c). 
\end{equation}
As a Levi subgroup of $\GLU_n$, the torus $T_n$ is indexed by $\{(n, 1)\}$, so it follows from Theorem~\ref{thm:ennola} and the discussion that precedes it that the value of the Green function $\Grfn^{T_n}_{T_n}(1, 1)$ is given by the Green polynomial evaluation $Q^{\langle1\rangle}_{\langle1\rangle}( (\vep q)^n)$, and therefore that
\begin{equation}
\label{eq:Singer}
\chi^{\phi, \lambda}(c) = (-1)^{k + n + n/d} \cdot\vep^{d \cdot a(\lambda') + \lfloor n/2 \rfloor} \cdot
Q^{\langle1\rangle}_{\langle1\rangle}((\vep q)^n) \sum_{\theta \in \phi} \theta(c).    
\end{equation}
We now consider summing this expression over all $d$-orbits $\phi$ in $\charorbs$; to that end, define 
\[
S_d(c) := \sum_{\substack{\phi \in \charorbs: \\ |\phi| = d}}
\sum_{\theta \in \phi} \theta(c)
\quad
\text{and} 
\quad
S'_d(c) := \sum_{m \mid d} S_m(c) = \sum_{\substack{\phi \in \charorbs: \\ |\phi| \text{ divides } d}}
\sum_{\theta \in \phi} \theta(c).
\]
By number-theoretic M\"obius inversion, we have 
\begin{equation}\label{eq:mobius inversion}
S_d(c) = \sum_{m \mid d} \mu(d/m) S'_m(c)
\end{equation}
for all divisors $d$ of $n$.  We now focus our attention on the computation of $S'_d(c)$.  
Suppose $c$ is a Singer cycle for $G_X$, with $X$ a multiplicative subgroup of $\GLU_1$ of order $|X| = x$.  Then $c$ has order $x \cdot \frac{q^n - \vep^n}{q - \vep}$, and is the $\big(\frac{q - \vep}{x}\big)$th power of a $\GLU_n$-Singer cycle $c'$.  Therefore $\theta(c) = (\theta(c'))^{(q - \vep)/x}$ for all irreducible characters $\theta$ of $T_n$.  Each such character is determined by the value $\theta(c') = \zeta$, where $\zeta$ is an $(q^n - \vep^n)$th root of unity.
A character $\theta \in \wh{\KK}^\times$ has $\Frob$-orbit size dividing $d$ if and only if $\Frob^d(\theta) = \theta$.  That is, the characters in the sum are those for which $\theta(z) = \theta\bigl(z^{(\vep q)^d}\bigr)$ for all $z \in T_n$, or equivalently for which $\theta\bigl(z^{q^d - \vep^d}\bigr) = 1$.  Thus, they are precisely the characters for which 
$\zeta = \theta(c')$ is actually a $(q^d - \vep^d)$th root of unity.  Therefore 
\[
S'_d(c)  = \sum_{\zeta \in \CC\colon \zeta^{q^d - \vep^d} = 1} \zeta^{(q - \vep)/x}.
\]
Since $\frac{q - \vep}{x} \mid q^d - \vep^d$, this can be rewritten as
\[
S'_d(c)  = \frac{q - \vep}{x} \cdot \Bigg( \sum_{\zeta \in \CC \colon \zeta^{ x \cdot \frac{q^d - \vep^d}{q - \vep}} = 1} \zeta \Bigg)
= \begin{cases}
    \frac{q - \vep}{x} & \text{ if } x \cdot \frac{q^d - \vep^d}{q - \vep} = 1, \\
    0 & \text{o.w.}
\end{cases}
\]
Of these two possibilities, the second holds in almost all cases: in order to have $x \cdot \frac{q^d - \vep^d}{q - \vep} = 1$, it must be that $x = 1$ and $q^d - \vep^d = q - \vep$.  This happens when $d = 1$ (for any $q$ and $\vep$) and when $q = d = 2$, $\vep = -1$, but not otherwise (since if $q > 2$ or $d > 2$ then $q^d - q > 2 \geq \vep^d - \vep$).  Moreover, the combination $d = 2$, $\vep = -1$ does not occur, since in this case $n$ is even and there is no Singer cycle $c$ in $\SL^{-}_n = \SU_n(q)$.  Therefore, we conclude
\[
S'_d(c) = \begin{cases}
    \frac{q - \vep}{x} & \text{ if } x = d = 1, \\
    0 & \text{o.w.},
\end{cases}
\]
and furthermore, by \eqref{eq:mobius inversion}, that
\[
S_d(c) = \begin{cases}
    (q - \vep) \cdot \mu(d) & \text{ in } \SLU_n, \\
    0 & \text{ in all other groups}.
\end{cases}
\]
The result follows immediately after combining this with \eqref{eq:Singer}.
\end{proof}

\subsection{Character values on transvections}

As discussed in Section~\ref{sec:conjugacy classes}, the transvections are the unipotent elements with Jordan form indexed by the partition $\langle 2, 1^{n-2}\rangle$.  Let $t$ be a transvection, let $\chi = \chi^{\phi, \lambda}$ be a primary irreducible character, and let $d= |\phi|$, so that $\lambda \vdash n/d$.  
Since $t$ is unipotent, in the Jordan decomposition $t = su$ we have $s = 1$ and $u = t$.  Thus, by \eqref{eq:primary char} and \eqref{eq:induction}, we have
\begin{multline*}
\chi^{\phi, \lambda}(t) = (-1)^{n + n/d} \cdot\vep^{d \cdot a(\lambda') + \lfloor n/2 \rfloor} \cdot
\sum_{\gamma \vdash n/d} \frac{\omega^\lambda(\gamma)}{z_\gamma \cdot |\TTT_{d\gamma}| \cdot |C_{\GLU_n}(1)|} 
\times{} 
\\
{}\times
\sum_{h \in \GLU_n: 1 \in {^h \TTT_{d\gamma}}} |C_{^h\TTT_{d\gamma}}(1)| 
\sum_{v \in \left(C_{^h\TTT_{d\gamma}}(1)\right)_\text{unip.}} \Grfn^{C_{\GLU_n}(1)}_{C_{^h\TTT_{d\gamma}}(1)}(t, v^{-1}) \cdot ({^h \theta})(v).
\end{multline*}
Since $\TTT_{d\gamma}$ consists exclusively of semisimple elements, the innermost sum has a unique term $v = 1$; and since all centralizers that appear in this formula are of the identity, it simplifies to
\begin{multline*}
\chi^{\phi, \lambda}(t) = (-1)^{n + n/d} \cdot\vep^{d \cdot a(\lambda') + \lfloor n/2 \rfloor} \cdot
 \sum_{\gamma \vdash n/d} \frac{\omega^\lambda(\gamma)}{z_\gamma \cdot |\TTT_{d\gamma}| \cdot |\GLU_n|} 
\times{} \\ {}\times 
 \sum_{h \in \GLU_n} |\TTT_{d\gamma}| \cdot \Grfn^{\GLU_n}_{^h\TTT_{d\gamma}}(t, 1) \cdot ({^h \theta})(1).
\end{multline*}
Since $\theta$ is an irreducible character of an abelian group $\TTT_{d\gamma}$, it has dimension $1$, so $({^h\theta})(1) = 1$ and 
\[
\chi^{\phi, \lambda}(t) = (-1)^{n + n/d} \cdot\vep^{d \cdot a(\lambda') + \lfloor n/2 \rfloor} \cdot
 \sum_{\gamma \vdash n/d} \frac{\omega^\lambda(\gamma)}{z_\gamma \cdot |\GLU_n|} \sum_{h \in \GLU_n} \Grfn^{\GLU_n}_{^h\TTT_{d\gamma}}(t, 1).
\]
By Theorem~\ref{thm:ennola} and the discusson that precedes it, the Green function evaluation $\Grfn^{\GLU_n}_{^h\TTT_{d\gamma}}(t, 1)$ is equal to $Q^{\langle 2, 1^{n - 2}\rangle}_{d\gamma}( \vep q)$, and so
\begin{equation}
\label{eq:transvections}
\chi^{\phi, \lambda}(t) =(-1)^{n + n/d} \cdot\vep^{d \cdot a(\lambda') + \lfloor n/2 \rfloor} \cdot
 \sum_{\gamma \vdash n/d} \frac{\omega^\lambda(\gamma)}{z_\gamma } Q^{\langle 2, 1^{n - 2}\rangle}_{d\gamma}( \vep q).
\end{equation}
This immediately leads to the following formula.
\begin{cor}\label{cor:transvections}
Fix a positive integer $n$, divisor $d \mid n$, and partition $\lambda \vdash n/d$.  Let
\[
R_{n, \lambda}(z) := \frac{(z^n - 1)(z^{n - 1} - 1)}{z - 1} \cdot \sum_{\gamma \vdash n/d} \frac{\omega^\lambda(\gamma)}{z_\gamma } Q^{\langle 2, 1^{n - 2}\rangle}_{d\gamma}( z).
\]
Then for any sign $\vep$ and any character orbit $\phi \in \charorbs$ of size $d$, we have
\[
\sum_{t \text{ transvection in } \GLU_n} \chi^{\phi, \lambda}(t) = (-1)^{n + n/d} \cdot\vep^{d \cdot a(\lambda') + \lfloor n/2 \rfloor} \cdot
 R_{n, \lambda}(\vep q).
\]
\end{cor}
\begin{proof}
Since the transvections form a single conjugacy class in $\GLU_n$, we have for each character $\chi$ that 
\[
\sum_{t \text{ transv.\ in } \GLU_n}\chi(t) = \#\{\text{transvections in } \GLU_n\} \cdot \chi(t)
\]
where on the right side $t$ represents any single transvection.  As mentioned in Section~\ref{sec:centralizers}, the number of transvections in $\GLU_n$ is 
\[
\left. \frac{(z^n - 1)(z^{n - 1} - 1)}{z - 1} \right|_{z \mapsto \vep q},
\]
and the result follows after combining this with \eqref{eq:transvections}.
\end{proof}

\subsection{Ennola duality in the determinant-\texorpdfstring{$1$}{1} case}

In this section, we prove a first part of our main theorems: the Ennola duality in the special groups.  Suppose $|X| = 1$, so that $G_X = \SLU_n$ and the only reflections in $G_X$ are the transvections.

\begin{theorem}\label{thm:SLU Ennola}
    Fix an odd positive integer $n$ and a nonnegative integer $\ell$.  There is a rational function $P := P_{n, \ell}$ with the property that for any prime power $q$, the number of reflection factorizations of length $\ell$ of a Singer cycle in $\SLnFq$ is $P(q)$ and the number of reflection factorizations of length $\ell$ of a Singer cycle in $\SU_n(q)$ is $P(-q)$.
\end{theorem}
\begin{proof}
Let $c$ be a Singer cycle in $\SLU_n$. By Corollary~\ref{cor:transvections}, we have that
\[
\sum_{\substack{t \text{ a reflec-} \\\text{tion in } \SLU_n}}  \hspace{-1em} \chi^{\phi, \lambda}(t) = 
\sum_{\substack{t \text{ a transvec-}\\\text{tion in } \GLU_n}} \hspace{-1em} \chi^{\phi, \lambda}(t) = (-1)^{n + n/d} \cdot\vep^{d \cdot a(\lambda') + \lfloor n/2 \rfloor} \cdot
 R_{n, \lambda}(\vep q).
\]
Up to the sign, this is an evaluation of a polynomial function at $\vep q$.  Moreover, while both the sign and the polynomial depend on the size $d$ of the orbit $\phi$ (as well as on the partition $\lambda$), they do not otherwise depend on $\phi$.
Similarly, we have by Proposition~\ref{prop:degrees} that the degree
\[
\deg(\chi^{\phi, \lambda}) = 
\vep^{d \cdot a(\lambda') + \binom{n}{2}} \cdot f^{d, \lambda}(\vep q)
\]
depends on $d$ but not otherwise on $\phi$ and is (up to the sign) an evaluation of a polynomial function at $\vep q$.
Thus, we can rewrite the formula provided by Corollary~\ref{cor:sparse} for the factorization counts in this case as
\begin{multline}\label{eq:SL-fac-first-simplification}
\Fac^{\{1\}}_\ell(c) = \frac{1}{|\GLU_n|} \sum_{d \mid n}  \sum_{k = 0}^{n/d - 1} \left(
\vep^{d \cdot \binom{n/d - k}{2} + \binom{n}{2}} \cdot f^{d, \hook{k}{n/d}}(\vep q)
\right)^{1 - \ell}  
 \times{}
 \\
 {}\times
\left( 
(-1)^{n + n/d} \cdot\vep^{d \cdot\binom{n/d - k}{2} + \lfloor n/2 \rfloor}
\cdot R_{d, \hook{k}{n/d}}(\vep q)\right)^\ell 
\cdot 
\sum_{\substack{\phi \in \charorbs: \\ |\phi| = d}} \chi^{\phi, \hook{k}{n/d}}(c^{-1}).
\end{multline}

As observed in the proof of Corollary~\ref{cor:sparse}, the element $c^{-1}$ is also a Singer cycle of $\SLU_n$.  Therefore, we can apply Proposition~\ref{prop:Singer sum} to the innermost sum in \eqref{eq:SL-fac-first-simplification} to get that
\begin{multline}\label{eq:SL-fac-second-simplification}
\Fac^{\{1\}}_\ell(c) = \frac{1}{|\GLU_n|} \sum_{d \mid n}  \sum_{k = 0}^{n/d - 1} \left(
\vep^{d \cdot \binom{n/d - k}{2} + \binom{n}{2}} \cdot f^{d, \hook{k}{n/d}}(\vep q)
\right)^{1 - \ell}   
 \times{}
 \\
 {}\times
\left( 
(-1)^{n + n/d} \cdot\vep^{d \cdot\binom{n/d - k}{2} + \lfloor n/2 \rfloor}
\cdot R_{d, \hook{k}{n/d}}(\vep q)\right)^\ell 
 \times{}
 \\
 {}\times
(-1)^{k + n + n/d} \cdot\vep^{d \cdot \binom{n/d - k}{2} + \lfloor n/2 \rfloor + 1}
\cdot Q^{\langle1\rangle}_{\langle1\rangle}((\vep q)^n) \cdot (\vep q - 1) \cdot \mu(d) \\
= \frac{\vep^{n + 1}}{ (\vep q)^{\binom{n}{2}} \prod_{i = 1}^n (( \vep q)^i - 1)} 
\sum_{d \mid n}  \sum_{k = 0}^{n/d - 1} \left(
 f^{d, \hook{k}{n/d}}(\vep q)
\right)^{1 - \ell}   
 \times{}
 \\
 {}\times
\left( 
(-1)^{n + n/d} 
\cdot R_{d, \hook{k}{n/d}}(\vep q)\right)^\ell 
 \cdot
(-1)^{k + n + n/d} 
\cdot Q^{\langle1\rangle}_{\langle1\rangle}((\vep q)^n) \cdot (\vep q - 1) \cdot \mu(d)
\end{multline}
for some polynomials $R_{d, \hook{k}{n/d}}$, $f^{d, \hook{k}{n/d}}$, and $Q^{\langle1\rangle}_{\langle1\rangle}$.  
Since $n$ is odd, $\vep^{n + 1} = 1$, and the result follows immediately.
\end{proof}

\subsection{Explicit enumeration in the determinant-\texorpdfstring{$1$}{1} case}

In the case $n = 1$, the group $\SLU_n$ is trivial; rather than worrying about whether its unique element is simultaneously the identity, a reflection, and a Singer cycle, we exclude that case from the theorem statement.

\begin{theorem}\label{thm:SL-enumeration-arbitrary-length}
    Fix integers $1 < n \leq \ell$.  If $n$ and $\vep$ are such that $\SLU_n$ contains a Singer cycle, then the number of factorizations of a Singer cycle in $\SLU_n$ as a product of $\ell$ reflections is
    \[
    [n]_{\vep q}^{\ell - 1} \cdot \sum_{i = 0}^{\ell - n} (-1)^i \binom{\ell}{i} \qbin{\ell - i - 1}{n - 1}{\vep q}.
    \]
\end{theorem}

\begin{proof}
We begin with the special linear group (i.e., in the case that $\vep = +$), starting from Proposition~\ref{cor:sparse}.\footnote{One could also carry out similar calculations beginning from \eqref{eq:SL-fac-second-simplification}; we instead take advantage of explicit character values that already exist in the literature.}
By \cite[Eq.~(5.5)]{LRS}, we have
\[
\sum_{t \text{ refn.~in } \SLnFq} \chi^{\phi, \hook{k}{n/d}}(t) = \deg(\chi^{\phi, \hook{k}{n/d}}) \cdot [n]_q \cdot \begin{cases} q^{n - k - 1} - 1 & \text{ if } d = 1 \\ -1 & \text{ if } d > 1\end{cases}.
\]
By \cite[III, 7, Ex.~1]{Macdonald}, the Green polynomial $Q^{\langle1\rangle}_{\langle1\rangle}(q)$ is the constant polynomial $1$, so by Proposition~\ref{prop:Singer sum}, 
\[
\sum_{\phi \in \charorbs: |\phi| = d} \chi^{\phi, \hook{k}{n/d}}(c)
= (-1)^{k + n + n/d} 
\cdot (q - 1) \cdot \mu(d).
\]
Plugging these values into the formula from Proposition~\ref{cor:sparse} gives
\begin{multline*}
\Fac^{\{1\}}_\ell(c) = 
\frac{1}{|\GL^+_n|} \sum_{d \mid n}  \sum_{k = 0}^{n/d - 1} f^{d, \hook{k}{n/d}}\left( q\right)
 \times{}
 \\
 {}\times
\left( [n]_q \cdot \begin{cases} q^{n - k - 1} - 1 & \text{if } d = 1 \\ -1 & \text{o.w.}\end{cases} \right)^\ell 
\cdot (-1)^{k + n + n/d} \cdot (q - 1) \cdot \mu(d).
\end{multline*}
It is convenient at this point to introduce the \defn{$q$-Pochhammer symbol}
\[
(a; q)_n := (1 - a)(1 - aq)(1 - aq^2) \cdots (1 - aq^{n - 1}).
\]
With this notation, the formula for the degree $f^{d, \hook{k}{n/d}}(q)$ in Proposition~\ref{prop:degrees} can be written as
\[
f^{d, \hook{k}{n/d}}(q) = (-1)^{n - n/d}q^{d \cdot \binom{k + 1}{2}} \cdot  \frac{(q; q)_{n}}{(q^d; q^d)_{n/d}} \cdot \qbin{n/d - 1}{k}{q^d},
\]
and $|\GL^+_n| = (-1)^n q^{\binom{n}{2}} (q; q)_n$,
whence
\begin{multline*}
\Fac^{\{1\}}_\ell(c) = \frac{(q - 1)[n]_q^\ell }{(-1)^n q^{\binom{n}{2}}} \sum_{k = 0}^{n - 1} (-1)^{k} \frac{q^{\binom{k + 1}{2}}}{(q; q)_{n}} \qbin{n - 1}{k}{q} \left( q^{n - k - 1} - 1 \right)^\ell  + {} \\
 {} + \left( -1 \right)^\ell \frac{(q - 1)[n]_q^\ell }{(-1)^n q^{\binom{n}{2}}} \sum_{1 < d \mid n} \frac{\mu(d)}{(q^d; q^d)_{n/d}}
 \sum_{k = 0}^{\frac{n}{d} - 1} (-1)^{k} q^{d\binom{k + 1}{2}}\qbin{n/d - 1}{k}{q^d}. 
\end{multline*}
By the $q$-binomial theorem \cite[p.~25, Exer.~1.2(vi)]{GasperRahman},
\[
\sum_{k = 0}^{\frac{n}{d} - 1} (-1)^{k} q^{d\binom{k + 1}{2}} \qbin{n/d - 1}{k}{q^d} = (q^d; q^d)_{\frac{n}{d} - 1}.
\]
Substituting this into the previous equation, using the identity $\sum_{1 < d \mid n} \mu(d) = -1$ for any integer $n > 1$, and simplifying gives 
\begin{align}
\hspace{15pt} & \hspace{-15pt} \Fac^{\{1\}}_\ell(c) = {} \notag \\
& = \frac{(q - 1)[n]_q^\ell }{(-1)^n q^{\binom{n}{2}}}
 \sum_{k = 0}^{n - 1} (-1)^{k} \frac{q^{\binom{k + 1}{2}}}{(q; q)_{n}} \qbin{n - 1}{k}{q} \left( q^{n - k - 1} - 1 \right)^\ell  +
  \frac{(q - 1)[n]_q^\ell }{(-1)^{\ell - n} q^{\binom{n}{2}}} \sum_{1 < d \mid n} 
  \frac{ \mu(d)}{(1 - q^n)} \notag \\
  & = \frac{(q - 1)[n]_q^\ell }{(-1)^n q^{\binom{n}{2}}}
 \sum_{k = 0}^{n - 1} (-1)^{k} \frac{q^{\binom{k + 1}{2}}}{(q; q)_{n}} \qbin{n - 1}{k}{q} \left( q^{n - k - 1} - 1 \right)^\ell  +
  \frac{(q - 1)[n]_q^\ell }{(-1)^{\ell - n} q^{\binom{n}{2}}}  \cdot 
  \frac{ -1}{(1 - q^n)}   \notag\\
& = \frac{[n]_q^{\ell - 1} }{ q^{\binom{n}{2}}}
 \sum_{k = 0}^{n - 1} (-1)^{n - k - 1} \frac{q^{\binom{k + 1}{2}}}{(q; q)_{n - 1}} \qbin{n - 1}{k}{q} \left( q^{n - k - 1} - 1 \right)^\ell  + (-1)^{\ell - n}
  \frac{[n]_q^{\ell - 1} }{q^{\binom{n}{2}}}. \label{eq:last formula}
\end{align}
We now show that this is equal to the desired answer.  Define the \defn{$q$-difference operator} $\Delta_q$, a linear operator acting on functions of a single variable $z$, via
\[
(\Delta_q f)(z) := \frac{f(z) - f(qz)}{z - qz}.
\]
It is easy to prove by induction (using the $q$-Pascal recurrence for $q$-binomial coefficients) that iteratively applying $\Delta_q$ $N$ times gives
\begin{equation}
\label{eq:q-diff-iterate-general}    
(\Delta_q^N f)(z) = \frac{1}{q^{\binom{N}{2}} z^N (q - 1)^N} \sum_{k = 0}^N (-1)^k q^{\binom{k}{2}} \qbin{n}{k}{q} f(q^{n - k}z),
\end{equation}
and also that 
\begin{equation}
\label{eq:q-diff-iterate-power}
\left.\Delta_q^N (z^a) \right|_{z = 1} = \frac{(q^{a - N + 1}; q)_N}{(1-q)^N}.
\end{equation}
Observe particularly that this latter expression is $0$ if $a$ is an integer and $0 \leq a < N$.

Now consider applying $\Delta_q^{n - 1}$ to the polynomial function $f(z) = \frac{(z -1)^\ell - (-1)^\ell}{z}$ and then evaluating at $z = 1$.  On one hand, we can write
\[
\left.\Delta_q^{n - 1}\left(\frac{(z - 1)^\ell - (-1)^\ell}{z}\right)\right|_{z = 1} = \left.\Delta_q^{n - 1}\left(\frac{(z - 1)^\ell}{z}\right)\right|_{z = 1} - \left.\Delta_q^{n - 1}\left(\frac{(- 1)^\ell}{z}\right)\right|_{z = 1}.
\]
The first term can be evaluated by \eqref{eq:q-diff-iterate-general}, yielding
\begin{align*}
\left.\Delta_q^{n - 1}\left(\frac{(z - 1)^\ell}{z}\right)\right|_{z = 1} 
& = 
\frac{1}{q^{\binom{n - 1}{2}} 1^{n - 1} (q - 1)^{n - 1}} \sum_{k = 0}^{n - 1} (-1)^k q^{\binom{k}{2}} \qbin{n - 1}{k}{q} \frac{(q^{n - 1 - k}\cdot 1 - 1)^\ell}{q^{n - 1 - k}\cdot 1} \\
& = \frac{1}{q^{\binom{n }{2}} (q - 1)^{n - 1}} \sum_{k = 0}^{n - 1} (-1)^k q^{\binom{k + 1}{2}} \qbin{n - 1}{k}{q} (q^{n - 1 - k} - 1)^\ell.
\end{align*}
The second term can be evaluated directly by \eqref{eq:q-diff-iterate-power}, yielding
\[
\left.\Delta_q^{n - 1}\left(\frac{(-1)^\ell}{z}\right)\right|_{z = 1} = (-1)^\ell \frac{(q^{-(n - 1)}; q)_{n - 1}}{(1 - q)^{n - 1}} = (-1)^{\ell} \frac{(q; q)_{n - 1}}{q^{\binom{n}{2}}(q - 1)^{n - 1}}.
\]
Comparing these two formulae with \eqref{eq:last formula}, we see that
\begin{equation}
\label{eq:penultimate SLn calculation}
\Fac^{\{1\}}_\ell(c) = 
\frac{[n]_q^{\ell - 1} (q - 1)^{n - 1}}{(-1)^{n - 1} (q; q)_{n - 1}} \cdot 
\left.\Delta_q^{n - 1}\left(\frac{(z - 1)^\ell - (-1)^\ell}{z}\right)\right|_{z = 1}.
\end{equation}
On the other hand, we can use the (usual) binomial theorem to expand $\frac{(z - 1)^\ell - (-1)^\ell}{z} = \sum_{i = 0}^{\ell - 1} (-1)^{i} \binom{\ell}{i} z^{\ell - i - 1}$; then applying \eqref{eq:q-diff-iterate-power} gives
\[
\left.\Delta_q^{n - 1}\left(\frac{(z - 1)^\ell - (-1)^\ell}{z}\right)\right|_{z = 1} =
\sum_{i = 0}^{\ell - 1} (-1)^{i} \binom{\ell}{i}
\frac{(q^{(\ell - i - 1) - (n - 1) + 1}; q)_{n - 1}}{(1-q)^{n - 1}}.
\]
The summand is $0$ if $\ell - i - 1 < n - 1$, so
\begin{align*}
\left.\Delta_q^{n - 1}\left(\frac{(z - 1)^\ell - (-1)^\ell}{z}\right)\right|_{z = 1} & =
\sum_{i = 0}^{\ell - n} (-1)^{i} \binom{\ell}{i}
\frac{(q^{(\ell - i - n + 1}; q)_{n - 1}}{(1-q)^{n - 1}} \\
& = \frac{(q; q)_{n - 1}}{(1 - q)^{n - 1}} \sum_{i = 0}^{\ell - n} (-1)^{i} \binom{\ell}{i} \qbin{\ell - i - 1}{n - 1}{q}.
\end{align*}
Combining this with \eqref{eq:penultimate SLn calculation} gives precisely the claimed formula in the case $\vep = +$.

We now consider the unitary case, so let $n$ be odd, $\vep = -$, and exclude the case $(n, q) = (3, 2)$.  By Theorem~\ref{thm:SLU Ennola}, there exists a rational function $P$ such that $P(q) = \Fac^{\{1\}}_\ell(c)$ for all prime powers $q$.  The formula just produced is manifestly rational (actually polynomial) in $q$.  A rational function can have only finitely many roots, so two rational functions that agree on infinitely many points are in fact equal as functions, and consequently the only possibility for $P$ is the formula just obtained.  Then the result for $\SU_n(q)$ follows from Theorem~\ref{thm:SLU Ennola}.
\end{proof}

\begin{cor}\label{cor:SL-enumeration-minimum-length}
    Fix an integer $n > 1$.  Suppose that $n$ and $\vep$ are such that $\SLU_n$ contains a Singer cycle $c$, and let $h = |c|$ be its multiplicative order.  Then the number of factorizations of $c$ in $\SLU_n$ as a product of $n$ reflections is $h^{n - 1}$.  In particular, this is $[n]_q^{n - 1} = \left(\dfrac{q^n - 1}{q - 1}\right)^{n - 1}$ when $\vep = +$ and is $[n]_{-q}^{n - 1} = \left(\dfrac{q^n + 1}{q + 1}\right)^{n - 1}$ when $n$ is odd and $\vep = -$.
\end{cor}
\begin{proof}
This follows immediately from Theorem~\ref{thm:SL-enumeration-arbitrary-length} upon substituting $\ell \mapsto n$, using the formulae for the orders of Singer cycles in Proposition~\ref{prop:Singer cycles exist}.
\end{proof}

Having completed the proofs of our main theorems in the determinant-$1$ case, we now move on to the remaining groups, all of which contain semisimple reflections.

\subsection{Character values on semisimple reflections}

As discussed in Section~\ref{sec:background}, the semisimple reflections are the diagonalizable elements of $\GLU_n$ in which $1$ is an eigenvalue of multiplicity $n - 1$.  Let $t$ be a semisimple reflection in $\GLU_n$, let $\chi^{\phi, \lambda}$ be a primary irreducible character for $\GLU_n$, and let $d = |\phi|$ (so that $\lambda \vdash n/d$). Since $t$ is semisimple, in its Jordan decomposition $t = su$, we have $s = t$ and $u = 1$. By \eqref{eq:primary char}, we have that the character value $\chi^{\phi, \lambda}(t)$ is a sum of (understandable multiples of) the Deligne--Lusztig character values $R^{\GLU_n}_{\TTT_{\gamma}}(\theta) (t)$.  Evaluating this via \eqref{eq:induction}, we are faced with a sum indexed by $\{h \in \GLU_n: t \in {^h \TTT_{d\gamma}}\}$.  By combining the description of the eigenvalues of Singer cycles (from Sections~\ref{sec:GL Singer cycles} and~\ref{sec:Singer cycles in other groups}) with the description of maximal tori in Section~\ref{sec:conjugacy classes}, we see that for any element of $\TTT_\mu$, the multiplicity of the eigenvalue $1$ is a sum of parts of $\mu$: each component $T_{\mu_i}$ either contributes the identity, with all eigenvalues $1$, or else contributes $\mu_i$ non-unit eigenvalues.  Therefore, $t$ belongs to a conjugate of the torus $\TTT_\mu$ if and only if $\mu$ has at least one of its parts equal to $1$.  This immediately implies the following statement.
\begin{prop}\label{prop:semisimple 0}
    For a semisimple reflection $t$ in $\GLU_n$ and a primary irreducible character $\chi^{\phi, \lambda}$ for $\GLU_n$, if $d := |\phi| > 1$ then $\chi^{\phi, \lambda}(t) = 0$.
\end{prop}

Now consider the case $d = 1$, so that $\lambda \vdash n$.  By \eqref{eq:primary char} and \eqref{eq:induction}, and using again that $t$ belongs to a conjugate of $\TTT_\gamma$ if and only if $\gamma$ has a part of size $1$, we have
\begin{multline*}
\chi^{\phi, \lambda}(t) = \vep^{a(\lambda') + \lfloor n/2 \rfloor}
\sum_{\substack{\gamma = \langle\gamma_1, \ldots, \gamma_k\rangle \vdash n \\ \gamma_k = 1}} 
\frac{\omega^\lambda(\gamma)}{z_\gamma \cdot |\TTT_\gamma| \cdot |C_{\GLU_n}(t)|}
\times {}
\\
{} \times 
\sum_{h \in \GLU_n: t \in {^h \TTT_\gamma}} |C_{^h\TTT_\gamma}(t)|
\sum_{v \in \left(C_{^h\TTT_\gamma}(t)\right)_\text{unip.}} \Grfn^{C_{\GLU_n}(t)}_{C_{^h\TTT_\gamma}(t)}(1, v^{-1}) \cdot ({^h \theta})(tv).
\end{multline*}
As before, $C_{^h\TTT_\gamma}(t) = {^h\TTT_\gamma}$ because the torus is abelian.  Furthermore, since $\TTT_\gamma$ consists entirely of semisimple elements, the innermost sum has a single term $v = 1$, and so
\[
\chi^{\phi, \lambda}(t) = \vep^{a(\lambda') + \lfloor n/2 \rfloor} \sum_{\substack{\gamma \vdash n \\ \gamma_k = 1}} 
\frac{\omega^\lambda(\gamma)}{z_\gamma \cdot |C_{\GLU_n}(t)|}
\sum_{h \in \GLU_n: t \in {^h \TTT_\gamma}} \Grfn^{C_{\GLU_n}(t)}_{^h\TTT_\gamma}(1, 1) \cdot ({^h \theta})(t). 
\]
As described above in Section~\ref{sec:conjugacy classes}, the centralizer $C_{\GLU_n}(t)$ consists of those elements of $\GLU_n$ that stabilize the two eigenspaces of $t$, and is isomorphic to $\GLU_{n - 1} \times \GLU_1$.  By Theorem~\ref{thm:ennola}, using the facts that $t$ is semisimple, the degrees of the indexing polynomials are $1$, and Green functions behave as one would expect on direct products, it follows that
\[
\Grfn^{C_{\GLU_n}(t)}_{^h \TTT_\gamma}(1, 1) = Q^{\langle1^{n - 1}\rangle}_{\gamma'}( \vep q) \cdot Q^{\langle1\rangle}_{\langle1\rangle}( \vep q)
\]
where $\gamma' = \langle\gamma_1, \ldots, \gamma_{k - 1}\rangle$ is the result of removing the last part (of size~$1$) from $\gamma$.  Thus
\begin{equation}
\label{eq:semisimple}
\chi^{\phi, \lambda}(t) = \vep^{a(\lambda') + \lfloor n/2 \rfloor} \sum_{\substack{\gamma \vdash n \\ \gamma_k = 1}} 
\frac{\omega^\lambda(\gamma) \cdot Q^{\langle1^{n - 1}\rangle}_{\gamma'}( \vep q) \cdot Q^{\langle1\rangle}_{\langle1\rangle}( \vep q)}{z_\gamma \cdot |\GLU_{n - 1}| \cdot |\GLU_1|}
\sum_{h \in \GLU_n: t \in {^h \TTT_\gamma}}  ({^h \theta})(t).
\end{equation}

In what follows, we will need to compute certain sums of character values on semisimple reflections that we describe now.
\begin{proposition}\label{prop:semisimple}
Fix a positive integer $n$, a sign $\vep$, and a partition $\lambda \vdash n$, and define
\[
S_\lambda(z) := \frac{z^{n - 1} \cdot (z^n - 1)}{z - 1} \cdot \sum_{\substack{\gamma \vdash n \colon \gamma \text{ has}\\\text{a part of size } 1}} \frac{\omega^\lambda(\gamma) \cdot  Q^{\langle1^{n - 1}\rangle}_{\gamma'}( z) \cdot Q^{\langle1\rangle}_{\langle1\rangle}( z) \cdot m_1(\gamma)}{z_\gamma},
\]
where $\gamma'$ is the result of removing a part of size $1$ from $\gamma$ and $m_1(\gamma)$ is the multiplicity of $1$ as a part of $\gamma$.
Let $\phi = \{\theta\} \in \charorbs$ be a $\Frob$-orbit of size $1$, with $\theta$ a character for $\GLU_1$, and let $X$ be a nontrivial subgroup of $\GLU_1$.

If the restriction $\theta \res_X$ is the trivial representation on $X$, then
\[
\sum_{\substack{t \text{ a semisimple}\\ \text{reflection in } G_X}} \chi^{\phi, \lambda}(t) = \vep^{a(\lambda') + \lfloor n/2 \rfloor} \cdot (|X| - 1) \cdot S_\lambda(\vep q).
\]
Otherwise,
\[
\sum_{\substack{t \text{ a semisimple}\\ \text{reflection in } G_X}} \chi^{\phi, \lambda}(t) = - \vep^{a(\lambda') + \lfloor n/2 \rfloor} \cdot S_\lambda(\vep q).
\]
\end{proposition}
\begin{proof}
Let $n, \vep, X, \lambda$ be as in the statement of the proposition, and let $\phi = \{\theta\} \in \charorbs$ be a $\Frob$-orbit of characters of size $1$.
Choose a semisimple reflection $t$ in $G_X$.  It shares its centralizer with all reflections that have the same fixed hyperplane and non-fixed eigenline.  Together with the identity, these reflections form a cyclic group $A$ that is isomorphic to $X$; in particular, there are $|X| - 1$ reflections in $A$, one of each determinant in $X$ other than $1$.  Obviously two different such subgroups $A_1$, $A_2$ intersect only at the identity.  Thus $\sum_{t \text{ s.s.}} \chi^{\phi, \lambda}(t)$ can be rewritten as $\sum_A \sum_{t\in A \smallsetminus \{1\}} \chi^{\phi, \lambda}(t)$, where the outer sum is over all such cyclic groups $A$.  By \eqref{eq:semisimple}, the inner sum can be rewritten as
\begin{multline*}
\sum_{t \in A \smallsetminus\{1\}} \chi^{\phi, \lambda}(t) = \vep^{a(\lambda') + \lfloor n/2 \rfloor} \sum_{\gamma} 
\frac{\omega^\lambda(\gamma) \cdot Q^{\langle1^{n - 1}\rangle}_{\gamma'}( \vep q) \cdot Q^{\langle1\rangle}_{\langle1\rangle}( \vep q)}{z_\gamma \cdot |\GLU_{n - 1}| \cdot |\GLU_1|}
\times {} \\ {} \times 
\sum_{t \in A \smallsetminus\{1\}} \sum_{h \in \GLU_n: t \in {^h \TTT_\gamma}}  ({^h \theta})(t).    
\end{multline*}

Consider now the inner double sum
\[
\sum_{t \in A \smallsetminus\{1\}} \sum_{h \in \GLU_n: t \in {^h \TTT_\gamma}}  ({^h \theta})(t).
\]
Suppose that $h \in \GLU_n$ is such that $t \in {^h \TTT_\gamma}$.  As discussed at the beginning of this subsection, the semisimple reflections in $\TTT_\gamma$ arise by writing $\TTT_\gamma \cong T_{\gamma_1} \times \cdots$ and taking the identity element in every component except for a single copy of $T_1 = \GLU_1$, while in that component we take the value $\det(t)$.  Two things immediately follow from this: first, that the entire cyclic group $A \ni t$ satisfies $A \subseteq {^h \TTT_\gamma}$ (for the same element $h$), and second, that the \emph{number} of conjugates of $t$ in $\TTT_\gamma$ is precisely the number $m_1(\gamma)$ of parts of $\gamma$ that are equal to $1$.  
By combining these two observations and denoting (in an abuse of notation) by $X$ the (unique) subgroup of $\GLU_1$ of order $|X|$, we immediately conclude that
\[
\sum_{t \in A \smallsetminus\{1\}} \sum_{h \in \GLU_n: t \in {^h \TTT_\gamma}}  ({^h \theta})(t)
=
m_1(\gamma) \cdot |C_{\GLU_n}(t)| \cdot \sum_{x \in X \smallsetminus\{1\}} \theta(x).
\]
Therefore
\begin{equation}\label{eq:another semisimple sum}
\sum_{t \in A \smallsetminus\{1\}} \chi^{\phi, \lambda}(t) = \vep^{a(\lambda') + \lfloor n/2 \rfloor} \sum_{\gamma} 
\frac{\omega^\lambda(\gamma) \cdot Q^{\langle1^{n - 1}\rangle}_{\gamma'}( \vep q) \cdot Q^{\langle1\rangle}_{\langle1\rangle}( \vep q) \cdot 
m_1(\gamma)}{z_\gamma}
 \cdot 
\sum_{x \in X \smallsetminus\{1\}} \theta(x).
\end{equation}

For any cyclic group $C$ and any one-dimensional character $\theta$ of $C$, we have
\[
\sum_{x \in C} \theta(x) = \begin{cases} |C| & \text{if } \theta = \one_C \text{ is the trivial character}, \\
0 & \text{otherwise.}
\end{cases}
\]
Applying this to the innermost sum in \eqref{eq:another semisimple sum} in the specific case that $C = X$ yields
\begin{multline*}
\sum_{t \in A \smallsetminus\{1\}} \chi^{\phi, \lambda}(t) = \vep^{a(\lambda') + \lfloor n/2 \rfloor} \sum_{\gamma} 
\frac{\omega^\lambda(\gamma) \cdot Q^{\langle1^{n - 1}\rangle}_{\gamma'}( \vep q) \cdot Q^{\langle1\rangle}_{\langle1\rangle}( \vep q) \cdot m_1(\gamma)}{z_\gamma}
\times {} \\ {} \times 
\begin{cases} |X| - 1 & \text{if } \phi = \{\theta\} \text{ where } \theta\res_{X} = \one_X \\ -1 & \text{otherwise}\end{cases}.
\end{multline*}
Since these sums (for different cyclic groups $A$) are pairwise disjoint, we can sum over all of them to conclude
\begin{multline}\label{eq:yet another ss sum}
\sum_{t \text{ s.s.\ refn.\ in } G_X}\chi^{\phi, \lambda}(t) = \vep^{a(\lambda') + \lfloor n/2 \rfloor} \frac{\text{\# s.s.\ refns.\ in } G_X}{|X| - 1} 
\times {} \\ {} \times 
\sum_{\gamma} \frac{\omega^\lambda(\gamma) \cdot Q^{\langle1^{n - 1}\rangle}_{\gamma'}( \vep q) \cdot Q^{\langle1\rangle}_{\langle1\rangle}( \vep q) \cdot m_1(\gamma)}{z_\gamma}
\begin{cases} |X| - 1 & \text{if } \phi = \{\theta\}, \theta\res_{X} = \one_X \\ -1 & \text{otherwise}\end{cases}.
\end{multline}
Then the claim follows immediately from the observations that the number of semisimple reflections is $|X| - 1$ times the number in each conjugacy class, and this latter quantity is
\[
\frac{|\GLU_n|}{|\GLU_{n - 1}| \cdot |\GLU_1|} = \frac{(\vep q)^{n - 1} \cdot ((\vep q)^n - 1)}{\vep q - 1}.
\qedhere
\]
\end{proof}

\subsection{Completing the proofs of the main theorems}

We are now ready to complete the proofs of the main results stated in the introduction.  We begin with the relevant version of Ennola duality for the groups that strictly contain $\SLU_n$.

\begin{theorem}\label{thm:|X| > 1}
    Fix an odd positive integer $n$ and nonnegative integers $m < \ell$.  There is a rational function $P := P_{n, m, \ell}$ with the following property: for every prime power $q$, sign $\vep$, and nontrivial subgroup $X$ of $\GLU_1$, the number of reflection factorizations of a Singer cycle in $G_X$ as a product of $\ell$ reflections, of which precisely $m$ are transvections, is
    \[
    \frac{(|X| - 1)^{\ell - m} - (-1)^{\ell - m}}{|X|} \cdot \binom{\ell}{m} \cdot P( \vep q ).
    \]
\end{theorem}

We remark that the hypotheses on $\ell, m, n$ make sense: by Proposition~\ref{prop:Singer cycles exist}, the group $G_X$ contains a Singer cycle $c$ with $\det(c) \neq 1$, so that every reflection factorization of $c$ requires at least one semisimple factor.  In the case $\vep = +, q = 2$, the statement is vacuous (there are no nontrivial subgroups of $\GL_1(2) = \{1\}$).

\begin{proof}
Let $\ell, m, n, q, X$ be as in the statement, let $c$ be a Singer cycle in $G_X$, and let $\Fac^X_{m, \ell}(c)$ denote the number of factorizations of the specified kind.  Combining the argument that established Corollary~\ref{cor:sparse} with Proposition~\ref{prop:semisimple 0}, we have that
\begin{multline}\label{eq:simplified in det neq 1 case}
\Fac^X_{m, \ell}(c) = \frac{\binom{\ell}{m}}{|\GLU_n|} \sum_{\substack{\phi \in \charorbs: \\ |\phi| = 1}} \sum_{k = 0}^{n - 1} \Bigg(
\deg(\chi^{\phi, \hook{k}{n}})^{1 - \ell} \chi^{\phi, \hook{k}{n}}(c^{-1}) 
 \times {}
 \\
 {} \times
\left(\sum_{t \text{ transv.}}\chi^{\phi, \hook{k}{n}}(t)\right)^m
\cdot
\left(\sum_{t \text{ s.s.\ refn.\ in } G_X}\chi^{\phi, \hook{k}{n}}(t)\right)^{\ell - m} \Bigg).
\end{multline}
(Here the binomial coefficient $\binom{\ell}{m}$ accounts for the fact that we are allowing the $m$ transvections to be in any position in the factorization; we could instead fix the positions and the analysis would be the same except for this factor.)
By Proposition~\ref{prop:degrees}, Corollary~\ref{cor:transvections}, and Proposition~\ref{prop:semisimple}, we can rewrite \eqref{eq:simplified in det neq 1 case} as
\begin{multline}\label{eq:second simplification in det neq 1 case}
\Fac^X_{m, \ell}(c) = \\
 \frac{\binom{\ell}{m}}{|\GLU_n|} \sum_{\substack{\phi = \{\theta\}\in \charorbs: \\ \theta \res_X = \one_X}}  \sum_{k = 0}^{n - 1} \Bigg( 
\left(\vep^{\binom{n - k}{2} + \binom{n}{2}} f^{1, \hook{k}{n}}(\vep q)\right)^{1 - \ell} 
\cdot 
\chi^{\phi, \hook{k}{n}}(c^{-1}) 
\times {} \\ {} \times
\left(\vep^{\binom{n - k}{2} + \lfloor n/2\rfloor} R_{n, \hook{k}{n}}(\vep q)\right)^m \left(\vep^{\binom{n - k}{2} + \lfloor n/2\rfloor} (|X| - 1) S_{\hook{k}{n}}(\vep q)\right)^{\ell - m} \Bigg)
+ {}
\\
{} +
\frac{\binom{\ell}{m}}{|\GLU_n|}\sum_{\substack{\phi = \{\theta\}\in \charorbs: \\ \theta \res_X \neq \one_X}} \sum_{k = 0}^{n - 1} \Bigg(
\left(\vep^{\binom{n - k}{2} + \binom{n}{2}} f^{1, \hook{k}{n}}(\vep q)\right)^{1 - \ell} \cdot 
\chi^{\phi, \hook{k}{n}}(c^{-1})
\times{}
\\
{} \times
\left(\vep^{\binom{n - k}{2} + \lfloor n/2\rfloor} R_{n, \hook{k}{n}}(\vep q)\right)^m \left(-\vep^{\binom{n - k}{2} + \lfloor n/2\rfloor} S_{\hook{k}{n}}(\vep q)\right)^{\ell - m} \Bigg),
\end{multline}
where $f$, $R$, and $S$ are polynomials with no dependence on $\phi$.
Therefore, after cancelling signs, switching the order of summation, and combining like terms, we have
\begin{multline}\label{eq:third simplification in det neq 1 case}
\Fac^X_{m, \ell}(c) = 
 \frac{\binom{\ell}{m}}{|\GLU_n|} 
 \sum_{k = 0}^{n - 1} 
 \vep^{\binom{n - k}{2} + \binom{n}{2}} \cdot 
 f^{1, \hook{k}{n}}(\vep q)^{1 - \ell} \cdot
 R_{n, \hook{k}{n}}(\vep q)^m \cdot
 S_{\hook{k}{n}}(\vep q)^{\ell - m}
 \times {}
 \\
 {} \times
 \Bigg( (|X| - 1)^{\ell - m} \sum_{\substack{\phi = \{\theta\}\in \charorbs: \\ \theta \res_X = \one_X}} \chi^{\phi, \hook{k}{n}}(c^{-1}) +
 (-1)^{\ell - m}\sum_{\substack{\phi = \{\theta\}\in \charorbs: \\ \theta \res_X \neq \one_X}} \chi^{\phi, \hook{k}{n}}(c^{-1}) \Bigg).
\end{multline}
By Proposition~\ref{prop:Singer sum}, since $X$ is nontrivial, the two sums in the innermost parenthesized sum in \eqref{eq:third simplification in det neq 1 case} are negatives of each other, and so
\begin{multline}\label{eq:fourth simplification in det neq 1 case}
\Fac^X_{m, \ell}(c) = 
 \frac{\binom{\ell}{m}}{|\GLU_n|} 
 \sum_{k = 0}^{n - 1} 
 \vep^{\binom{n - k}{2} + \binom{n}{2}} \cdot 
 f^{1, \hook{k}{n}}(\vep q)^{1 - \ell} \cdot
 R_{n, \hook{k}{n}}(\vep q)^m \cdot
 S_{\hook{k}{n}}(\vep q)^{\ell - m}
 \times {}
 \\
 {} \times
 \left( (|X| - 1)^{\ell - m}  - (-1)^{\ell - m}\right) \cdot \sum_{\substack{\phi = \{\theta\}\in \charorbs: \\ \theta \res_X = \one_X}} \chi^{\phi, \hook{k}{n}}(c^{-1}).
\end{multline}
When $\theta \in \wh{T}_1$ is such that $\theta \res_{X} = \one_X$, we have by \eqref{eq:Singer} that
\[
\chi^{\{\theta\}, \hook{k}{n}}(c^{-1}) = (-1)^k \vep^{\binom{n - k}{2} + \lfloor n/2 \rfloor} Q^{\langle1\rangle}_{\langle1\rangle}((\vep q)^n).
\]
Moreover, the number of such characters on $\GLU_1$ is precisely $\frac{|\GLU_1|}{|X|} = \vep \cdot \frac{ \vep q - 1}{|X|}$.  Therefore \eqref{eq:fourth simplification in det neq 1 case} simplifies further, to
\begin{multline*}
\Fac^X_{m, \ell}(c) = \left( \left(|X| - 1\right)^{\ell - m} - (-1)^{\ell - m}\right) \cdot \frac{\binom{\ell}{m}}{|\GLU_n|}  \sum_{k = 0}^{n - 1} \Bigg(
  f^{1, \hook{k}{n}}(\vep q)^{1 - \ell} 
\times{}
\\
{} \times
 R_{n, \hook{k}{n}}(\vep q)^m \cdot S_{\hook{k}{n}}(\vep q)^{\ell - m} \cdot
Q^{\langle1\rangle}_{\langle1\rangle}((\vep q)^n) \cdot (-1)^k \cdot \vep \cdot \frac{\vep q - 1}{|X|}\Bigg).
\end{multline*}
Moving the factor $\frac{\vep}{|X|}$ outside the sum, expanding $|\GLU_n| = \vep^n (-1)^n (\vep q)^{\binom{n}{2}} (\vep q ; \vep q)_n$, combining the powers of $\vep$, and observing that (since $n$ is odd) $\vep^{n + 1} = 1$ gives the result.
\end{proof}

\begin{cor}\label{cor:counting |X| > 1}
Fix a positive integer $n$, a sign $\vep$, a prime power $q$, and a nontrivial subgroup $X$ of $\GLU_1$.  Suppose that these choices are such that the group $G_X \leq \GLU_n$ contains a Singer cycle $c$.  Then for nonnegative integers $m < \ell$, the number of factorizations of $c$ as a product of $\ell$ reflections, of which exactly $m$ are transvections, is
\[
\binom{\ell}{m} \cdot \frac{(|X| - 1)^{\ell - m} - (-1)^{\ell - m}}{|X|} \cdot [n]_{\vep q}^{\ell - 1} \cdot \sum_{i = 0}^{\min(m, \ell - n)} (-1)^i \binom{m}{i} \qbin{\ell - i - 1}{n - 1}{\vep q}.
\]
\end{cor}
\begin{proof}
From \cite[Thm.~1.3]{LRS} we have when $\vep = +$ and $q > 2$ that 
\[
\Fac^{\Fq^\times}_{m, \ell}(c) = \binom{\ell}{m} \cdot \frac{(q - 2)^{\ell - m} - (-1)^{\ell - m}}{q - 1} \cdot [n]_q^{\ell - 1} \cdot \sum_{i = 0}^{\min(m, \ell - n)} (-1)^i \binom{m}{i} \cdot \qbin{\ell - i - 1}{n - 1}{q}
\]
(because there are exactly $\binom{\ell}{m} \cdot \frac{(q - 2)^{\ell - m} - (-1)^{\ell - m}}{q - 1}$ sequences of $\ell$ elements in $\Fq^\times$ that contain $m$ copies of $1$ and have product $\det(c)$, a fixed value other than $1$).
It follows immediately that the rational function $P_{n, m, \ell}$ in Theorem~\ref{thm:|X| > 1} satisfies
\begin{equation}\label{eq:polynomial for |X| > 1}
P_{n, m, \ell}(q) = [n]_q^{\ell - 1} \cdot \sum_{i = 0}^{\min(m, \ell - n)} (-1)^i \cdot \binom{m}{i} \cdot \qbin{\ell - i - 1}{n - 1}{q}
\end{equation}
for all prime powers $q > 2$, and therefore that this is an equality of polynomials.
\end{proof}

\begin{cor}\label{cor:general-enumeration-minimal-length}
Fix a positive integer $n$, a sign $\vep$, a prime power $q$, and a nontrivial subgroup $X$ of $\GLU_1$.  Suppose that these choices are such that the group $G_X \leq \GLU_n$ contains a Singer cycle $c$, of order $h = |X| \cdot \frac{(\vep q)^n - 1}{\vep q - 1}$.  Then the number of factorizations of $c$ as a product of the minimum number $n$ of reflections is $h^{n - 1}$.
\end{cor}
\begin{proof}
By Corollary~\ref{cor:counting |X| > 1}, the desired number of factorizations is
\begin{align*}
[n]_{\vep q}^{n - 1}\sum_{m = 0}^{n - 1} \binom{n}{m} \frac{(|X| - 1)^{n - m} - (-1)^{n - m}}{|X|} & = \frac{[n]_{\vep q}^{n - 1}}{|X|} \cdot \left(\left((|X| - 1) + 1\right)^n - ((-1) + 1)^n\right) \\
& \overset{(n > 0)}{=} [n]_{\vep q}^{n - 1} \cdot |X|^{n - 1},
\end{align*}
as claimed.
\end{proof}

We now collect the preceding results into the proofs of the three main theorems in the introduction.

\begin{proof}[Proof of Main Theorems~\ref{main:short} and~\ref{main:ennola}]
In the case that $G = \SLU_n$, these results are precisely the content of Corollary~\ref{cor:SL-enumeration-minimum-length}.  In the case of groups $G$ such that $\SLU_n < G \leq \GLU_n$, the enumerative formula in Main Theorem~\ref{main:short} is given by Corollary~\ref{cor:general-enumeration-minimal-length}, while the fact that this is an Ennola duality for $\GLU_n$ follows from this after making the substitution $|X| = |\GLU_1| = q - \vep$ and observing that the condition that Singer cycles exist in the unitary groups means that $n$ is odd, so the outer exponent $n - 1$ is even.
\end{proof}

\begin{proof}[Proof of Main Theorem~\ref{main:longer factorizations}]
In the case that $G = \SLU_n$, this result is precisely the content of Theorem~\ref{thm:SL-enumeration-arbitrary-length}.  In the other cases, we have by Corollary~\ref{cor:counting |X| > 1} that
\[
\Fac^X_\ell(c) = \sum_{m = 0}^{\ell - 1}\binom{\ell}{m} \cdot \frac{(|X| - 1)^{\ell - m} - (-1)^{\ell - m}}{|X|} \cdot [n]_{\vep q}^{\ell - 1} \cdot \sum_{i = 0}^{\min(m, \ell - n)} (-1)^i \binom{m}{i} \qbin{\ell - i - 1}{n - 1}{\vep q}.
\]
Reversing the order of summation, this becomes
\begin{align*}
\Fac^X_\ell(c) & = 
\frac{[n]_{\vep q}^{\ell - 1}}{|X|} 
\sum_{i = 0}^{\ell - n} (-1)^i \qbin{\ell - i - 1}{n - 1}{\vep q}
\sum_{m = i}^{\ell - 1} \binom{\ell}{m}\binom{m}{i} \left((|X| - 1)^{\ell - m} - (-1)^{\ell - m}\right) \\
& = \frac{[n]_{\vep q}^{\ell - 1}}{|X|} 
\sum_{i = 0}^{\ell - n} (-1)^i \qbin{\ell - i - 1}{n - 1}{\vep q}\binom{\ell}{i}
\sum_{m = i}^{\ell - 1} \binom{\ell - i}{\ell - m} \left((|X| - 1)^{\ell - m} - (-1)^{\ell - m}\right) \\
& = \frac{[n]_{\vep q}^{\ell - 1}}{|X|} 
\sum_{i = 0}^{\ell - n} (-1)^i \binom{\ell}{i}\qbin{\ell - i - 1}{n - 1}{\vep q} \left( (|X|^{\ell - i} - 1) - (0 - 1)\right) \\
& = [n]_{\vep q}^{\ell - 1} \cdot
\sum_{i = 0}^{\ell - n} (-1)^i \binom{\ell}{i}\qbin{\ell - i - 1}{n - 1}{\vep q} |X|^{\ell - i - 1},
\end{align*}
as claimed.
\end{proof}

\section{Further remarks and open questions}
\label{sec:final remarks}

We end with some open questions suggested by our work.

\subsection{Refining by eigenvalues}
In \cite[Thm.~1.3]{LRS}, a more refined version of Corollary~\ref{cor:counting |X| > 1} was given in the case of the group $\GLnFq$; it provides the enumeration of reflection factorizations of a Singer cycle in $\GLnFq$ where the factors have a specified sequence of eigenvalues.  It was observed in particular that the resulting enumeration depends only very slightly on the specified sequence of eigenvalues, namely, it is determined by the length and the number of transvection factors (provided that the product of the sequence of determinants agrees with the determinant of the Singer cycle being factored).  It is natural to conjecture that the same statement is true in the unitary case.

\begin{conjecture}
    Let $c$ be any irreducible element in $\GLU_n$ and let $(a_1, \ldots, a_\ell)$ be a sequence of elements of $\GLU_1$ such that $a_1 \cdots a_\ell = \det(c)$.  Suppose that precisely $m$ of the $a_i$ are equal to $1$.  Then the number of factorizations $c = t_1 \cdots t_\ell$ of $c$ as a product of $\ell$ reflections such that $\det(t_i) = a_i$ for $i = 1, \ldots, \ell$ is
    \[
    [n]_{\vep q}^{\ell - 1} \cdot \sum_{i = 0}^{\min(m, \ell - n)} (-1)^i \cdot \binom{m}{i} \cdot \qbin{\ell - i - 1}{n - 1}{\vep q}.
    \]
\end{conjecture}

This is consistent with Corollary~\ref{cor:counting |X| > 1}: in any finite group $G$, for any fixed element $G$ other than the identity, the number of tuples $(x_1, \ldots, x_k)$ of non-identity elements of $G$ such that $x_1 \cdots x_k = x$ is $\frac{(|G| - 1)^k - (-1)^k}{|G|}$, and Corollary~\ref{cor:counting |X| > 1} would follow immediately after combining this fact with the conjecture.

\subsection{Another approach to character values on reflections}

In \cite{HLR, LewisMorales}, an alternative strategy was employed for computing character values on reflections that considered
\[
\sum_{g : \dim \fix(g) = k} \chi(g)
\]
for all $k$.  This latter approach is more general, because $k$ need not be $n - 1$, but less refined, because it does not separate reflections by type (semisimple versus unipotent).  These character sums have simple, attractive formulae for many $\GLnFq$-characters.  Experimentation suggests that this phenomenon extends to the unitary case.
\begin{conjecture}
  Suppose $\chi = \chi^{\phi, \lambda}$ is a primary irreducible character for $\GUnFq$, where $\phi$ is any orbit other than $\{\one\}$ (where $\one$ is the trivial character of $\GU_1(q)$).  Then
    \[
    \sum_{g : \dim \fix(g) = k} \frac{\chi(g)}{\chi(1)} = (-1)^k (-q)^{\binom{k}{2}} \qbin{n}{k}{-q}.
    \]
\end{conjecture}
The formula on the right side of the conjecture is the result of plugging $-q$ in for $q$ in the corresponding formula in \cite[Prop.~4.10(i)]{HLR}.

\subsection{A factorization poset, and a connection with invariant theory?}

To each pair $(G, R)$ of a group $G$ and a generating set $R$, there is a naturally associated partial order (poset) structure (as in, e.g., \cite{HLR, MuhleRipoll, LW2}): for $g \in G$, one writes $\ell_R(g)$ for the \defn{$R$-length} of $g$, i.e., the smallest number $k$ such that $g = r_1 \cdots r_k$ with $r_1, \ldots, r_k \in R$, and one sets $g \leq h$ in the partial order if and only if $\ell_R(h) = \ell_R(g) + \ell_R(g^{-1}h)$.  In the case that $G$ is a finite Coxeter group and $R$ is the set of reflections in $G$, this poset is called the \defn{absolute order} on $G$, the interval $[1, c]$ in this poset between the identity and a Coxeter element $c$ is the \defn{lattice of $G$-noncrossing partitions}, and the factorization-counting theorems of Hurwitz--D\'enes--Chapoton--Bessis mentioned in the introduction give the number of maximal chains in this lattice.  In the real and complex cases, the numerology of the invariant theory of $G$ (degrees, exponents, etc.) is closely tied to the combinatorics of the lattice (rank sizes, chain counts, etc.)---see, for example, \cite{Armstrong, ChapuyDouvropoulos}.

In the case that $G = \GLnFq$ and $R$ is the set of reflections of $G$, various enumerative and structural properties of the poset were explored in \cite{HLR}.  Can the same be done for the other groups considered here?  How does the numerology compare to the numerology coming from invariant theory, as in \cite{HansonShepler}?

\subsection{Extension to other groups}

It is natural to ask to what extent the main theorem can be extended more generally, say to other finite groups of Lie type.  Huppert \cite{Huppert1970} showed that the finite symplectic groups $\op{Sp}_{2n}(q)$ and the finite minus-type orthogonal groups $\op{O}^-_{2n}(q)$ of even rank contain irreducible elements, ergo Singer cycles (and that the other orthogonal types do not).  Computational exploration in small cases suggests the following conjecture.

\begin{conjecture} \label{conj:symplectic and orthogonal}
    Suppose that $G$ is either the finite symplectic group $\op{Sp}_{N}(q)$ or orthogonal group $\op{O}^-_{N}(q)$ and that $c$ is a Singer cycle in $G$ (so $N = 2n$ is even).  Let $h := |c|$ be the multiplicative order of $c$.  Then the number of factorizations of $c$ as a product of the minimum number $N$ of reflections in $G$ is $h^{N - 1}$.
\end{conjecture}

Unfortunately, the approach we've taken above for the unitary and linear groups seems intractable in this case: the representation theory of $\op{Sp}_{2n}(q)$ and $\op{O}^{-}_{2n}(q)$ is much less explicit than that of $\GLU_n$, and there does not seem to be any analogue of the sparsity theorem Proposition~\ref{prop:mostly zero} in these types.

On the other hand, it might be possible to attack the orthogonal and symplectic cases by more elementary means.  In \cite{BradyWatt}, Brady and Watt show that if $V$ is a vector space and $\langle \cdot, \cdot\rangle$ is an anisotropic\footnote{That is, $\langle v, v \rangle = 0$ implies that $v$ is the zero vector.} nondegenerate symmetric bilinear form, then the set of factorizations of any orthogonal transformation $g$ as a product of the minimum number of orthogonal reflections is in bijection with the maximal flags of subspaces in the \emph{move space} $\mov(g) := \im(g - 1)$.  Over a finite field, it is too much to hope that a form might be anisotropic. However, one could try to follow the work of \cite{McCammondPaolini}, building on earlier work of Wall \cite{Wall}, to give a bijection between factorizations of a Singer cycle and appropriate linear-algebraic objects (certain special flags or ordered bases in $V$), and then to try to count these directly. 

If one could prove Conjecture~\ref{conj:symplectic and orthogonal} by such techniques, it would provide nontrivial information about character sums in the orthogonal and symplectic groups, which might be of interest on their own!

Alternatively, one could seek \emph{uniform} proofs that work simultaneously for all finite groups of Lie type (say).  In the case of real and complex reflection groups, 
this has been carried out by Michel \cite{Michel} (for Weyl groups) and Douvropoulos \cite{Douvropoulos} (in general), including the case of factorizations of arbitrary length (analogous to our Main Theorem~\ref{main:longer factorizations}, and originally proved case-by-case using the classification \cite{ChapuyStump}).  We have not explored to what extent similar ideas can be employed in our context.

\subsection{Factoring other irreducible elements}
\label{rmk:things okay for other irreducibles}

It is natural to ask to what extent the given results depend specifically on the fact that $c$ is a Singer cycle, rather than an arbitrary irreducible element.  The multiplicative order of $c$ is used explicitly in the proof of Proposition~\ref{prop:Singer sum}, but this is not essential: if $c$ is an irreducible element in $G_X$ that is the $r$th power of a $\GLU$-Singer cycle $c'$, then (continuing the notation of the proof of the proposition) the same argument implies that
\[
S'_d(c) = \sum_{\zeta \in \CC\colon \zeta^{q^d - \vep^d} = 1} \zeta^{r} = \begin{cases} q^d - \vep^d & \text{ if } q^d - \vep^d \mid r, \\ 0 & \text{ o.w.}\end{cases}
\]
However, if $q^d - \vep^d \mid r$ and $d > 1$, one can show (after a little case analysis) that $(c')^r$ is not irreducible.  Thus, we conclude that $S'_d(c) = 0$ unless $d = 1$ and $q - \vep \mid r$, i.e., unless $d = 1$ and $\det(c) = 1$.  Thus we conclude the following strengthened version of Proposition~\ref{prop:Singer sum}.
\begin{prop}
 Let $c$ be an irreducible element in $\GLU_n$. For each positive integer $d \mid n$ and each hook-shaped partition $\hook{k}{n/d}$ of $n/d$, we have
    \[
    \sum_{\substack{\phi \in \charorbs: \\ |\phi| = d}} \chi^{\phi, \hook{k}{n/d}}(c)
    =
    \begin{cases}
    \pm Q^{\langle1\rangle}_{\langle1\rangle}((\vep q)^n) \cdot (q - \vep) \cdot \mu(d) & \text{ if } d = 1 \text{ and } \det(c) = 1\\
    0 & \text{ otherwise},
\end{cases}
    \]
 with the same sign as in Proposition~\ref{prop:Singer sum}.
\end{prop}
Tracing through the rest of the arguments, we conclude (as in \cite[\S6.6]{LRS}) that if $c$ is irreducible and $\det(c) \neq 1$ then the number of reflection factorizations of $c$ in $G_X$ is the same as the number of factorizations of a $G_X$-Singer cycle.  In the case that $\det(c) = 1$ and $|X| > 1$, more work is necessary: the $\SLU_n$-factorizations, using only transvections, will be equinumerous with the factorizations of an $\SLU_n$-Singer cycle, but one must also count the factorizations that use semisimple factors and combine the results; we have not carried out the details, which are ``just a computation''.

\section*{Acknowledgements}
This project originated a decade ago, during the first-named author's postdoc under the mentorship of Vic Reiner and Dennis Stanton. In the ensuing years, the first-named author has had conversations with numerous people that affected his thinking about the topic; if you were one of these people, thank you.  We also thank an anonymous referee for helpful comments.  

The first-named author was supported in part by a gift from the Simons Foundation (MPS-TSM-00006960).  The second-named author was supported in part by a grant from the Simons Foundation (Award \#713090).

\bibliography{GLnFq}{}
\bibliographystyle{alpha}

\end{document}